\documentclass[ruled]{article}
\usepackage{geometry}

\usepackage[latin9]{inputenc}
\usepackage{color}
\usepackage{mathtools}
\usepackage{multibib}
\newcites{sup}{References}
\usepackage[authoryear]{natbib}
\usepackage{todonotes}
\usepackage{amsmath}
\usepackage{mathrsfs}
\usepackage{accents}
\usepackage{float}
\usepackage{graphicx}
\usepackage{caption}
\usepackage{subcaption}
\usepackage{enumitem}  
\usepackage[unicode=true,pdfusetitle, bookmarks=true,bookmarksnumbered=false,bookmarksopen=false, breaklinks=false,pdfborder={0 0 1},backref=false,colorlinks=true] {hyperref}
\usepackage{setspace}
\usepackage{amsthm}
\usepackage{bbm}
\usepackage{amssymb}

\hypersetup{linkcolor=blue,citecolor=blue}

\makeatletter

\renewcommand\paragraph{\@startsection{paragraph}{4}{\z@}%
            {-2.5ex\@plus -1ex \@minus -.25ex}%
            {1.25ex \@plus .25ex}%
            {\normalfont\normalsize\bfseries}}

\theoremstyle{plain}

\theoremstyle{plain}

\floatstyle{ruled}
\newfloat{algorithm}{tbp}{loa}
\providecommand{\algorithmname}{Algorithm}
\floatname{algorithm}{\protect\algorithmname}

\usepackage{algorithm}
\usepackage{dsfont}
\usepackage{algpseudocode}
\usepackage{mathtools}

\newcounter{daggerfootnote}

\newcommand{\setX}{\mathsf{X}}

\newcommand{\F}{\mathcal{F}}
\newcommand{\dd}{\mathrm{d}}
\newcommand{\bigO}{\mathcal{O}}
\newcommand{\R}{\mathbb{R}}
\renewcommand{\P}{\mathbb{P}}
\newcommand{\E}{\mathbb{E}}

\newcommand{\ind}{\mathbbm{1}}

\newtheorem{remark}{Remark}
\newtheorem{theorem}{Theorem}
\newtheorem{lemma}{Lemma}
\newtheorem{proposition}{Proposition}

\newtheorem{corollary}{Corollary}

\makeatother

\title{Safety of particle filters: Some results   on the time evolution of particle filter estimates}

\date{}
\author{}

\begin{document}
\maketitle

\begin{center}
\vspace{-1.3cm}
\large{Mathieu Gerber\\
\vspace{0.3cm} University of Bristol}\\
\vspace{0.5cm}

\end{center}

\begin{abstract}
Particle filters (PFs) form a class of Monte Carlo algorithms that  propagate  over time a set of $N\geq 1$ particles which can be used to estimate, in an online fashion, the sequence of filtering distributions $(\hat{\eta}_t)_{t\geq 1}$ defined by a state-space model.  Despite the  popularity  of PFs, the study of the time evolution of their estimates has   received barely any attention   in the literature. Denoting by $(\hat{\eta}_t^N)_{t\geq 1}$ the PF estimate of $(\hat{\eta}_t)_{t\geq 1}$ and letting $\kappa\in (0,1/2)$, in this work  we first show that for any number of particles  $N$  it holds that, with probability one, we have $\|\hat{\eta}_t^N- \hat{\eta}_t\|\geq \kappa$ for infinitely many time instants $t\geq 1$, with $\|\cdot\|$   the Kolmogorov distance between probability distributions. Considering a simple filtering problem we then provide reassuring results concerning the ability of PFs to estimate jointly a finite set $\{\hat{\eta}_t\}_{t=1}^T$ of filtering distributions by studying the probability $\P(\sup_{t\in\{1,\dots,T\}}\|\hat{\eta}_t^{N}-\hat{\eta}_t\|\geq \kappa)$. Finally,  on the same toy filtering problem, we prove that   sequential quasi-Monte Carlo, a  randomized quasi-Monte Carlo version of PF algorithms, offers greater safety guarantees than PFs in the sense that, for this algorithm,  it holds that $\lim_{N\rightarrow\infty}\sup_{t\geq 1}\|\hat{\eta}_t^N-\hat{\eta}_t\|=0$ with probability one.
\end{abstract}

\section{Introduction}

\subsection{Context\label{sub:context}}
 
Particle filters (PFs) are  used for real-time  inference in state-space models (SSMs) in various applications, including robotics \citep{thrun2002particle,hailu2024theories}, self-driving cars \citep{berntorp2019particle,hafez2024safe}, assisted surgery \citep{kummert2021efficient} and ballistics objects tracking \citep{kim2023accelerated}.

 Denoting by $t\in\mathbb{N}:=\{1,2,\dots\}$ the time index and by $\hat{\eta}_t$ the conditional distribution of the state variable at  time $t$ given the available observations $\{y_s\}_{s=1}^t$, a PF is a Monte Carlo algorithm that propagates over time a set of $N\in\mathbb{N}$ particles that can be used to estimate, in an online fashion, the sequence of the so-called filtering distributions $(\hat{\eta}_t)_{t\geq 1}$. For instance, in  self-driving cars applications $\hat{\eta}_t$ is the conditional distribution of the position of the car at time $t$ given  its  GPS localisation   (known to have a precision of a few meters) and   its distance   with respect to some nearby landmarks, under the assumed SSM.

The popularity of PFs stems from the fact that this class of algorithms  can be deployed on a very large class of SSMs. In addition, it is widely believed that the quality of the PF estimate $\hat{\eta}_t^N$ of $\hat{\eta}_t$ does not deteriorate over time, a belief  supported by several theoretical analyses  showing that,  for some class of functions $\mathfrak{F}$ and constant $p\in\mathbb{N}$,  we have \citep[see][and references therein]{caffarel2024mathematical}  
\begin{align}\label{eq:Bound_the}
\lim_{N\rightarrow\infty} \sup_{t\geq 1}\E\big[|\hat{\eta}_t^N(f)-\hat{\eta}_t(f)|^p\big]=0,\quad\forall f\in\mathfrak{F}.
\end{align}

However, this type of results only provides a guarantee for the error that arises when   estimating $\hat{\eta}_t$   by $\hat{\eta}_t^N$ at a single time instant $t$, while in practice we are usually interested in estimating accurately   $\hat{\eta}_t$ for several time instants $t\in\{1,\dots,T\}$; for instance, a PF used for positioning in a self-driving car needs to provide a precise localization of the vehicle during its whole life time. In this context, instead of \eqref{eq:Bound_the} a more relevant theoretical guarantee for PFs would be that, in some sense, $\sup_{t\geq 1}|\hat{\eta}_t^N(f)-\hat{\eta}_t(f)|\rightarrow 0$ as $N\rightarrow\infty$ and  for all functions $f$ belonging to some class of functions.

 \subsection{Contributions of the paper}

In this paper we first show that such  a  theoretical guarantee  for PFs cannot exist.  Formally, denoting by $\|\cdot\|$ the Kolmogorov distance between probability distributions, we prove that  for any $\kappa\in (0,1/2)$, with probability one and for any  number of particles $N\geq 1$, we have $\|\hat{\eta}_t^N- \hat{\eta}_t\|\geq \kappa$ for infinitely many time instants $t\geq 1$. The fact that in some challenging situations a PF may fail to accurately estimate  some filtering distributions is not surprising: it is well-known that this problem  will usually  happen when the distance between two successive filtering distributions $\hat{\eta}_{t}$ and $\hat{\eta}_{t+1}$ is large,  e.g.~because the SSM is not time homogenous or because the observations $y_{t}$ and $y_{t+1}$ are very different from each other. However, and crucially,  we prove the above property of PFs on what is arguably the simplest filtering scenario, namely  for a one-dimensional and time-homogenous linear Gaussian SSM with observations $y_t=0$ for all $t\geq 1$.  As we will see in what follows, this negative result for PF estimates is trivial to prove, and arises from the very simple and intuitive reason that too much randomness is used to propagate the particle system over time. We stress that this problem is not limited to PFs   but is   common and intrinsic to any  sequential (importance) sampling algorithms, including the ensemble Kalman filter \citep{roth2017ensemble}.

For any given $\kappa\in(0,1)$, the PF estimate  $(\hat{\eta}_t^N)_{t\geq 1}$ of the sequence $(\hat{\eta}_t)_{n\geq 1}$ is   therefore necessarily such that    $\P(\sup_{t\geq 1}\|\hat{\eta}_t^N-\hat{\eta}_t\|\geq \kappa)=1$. Given the ever-increasing use of PFs in life critical applications,  notably due to the increasing prevalence of autonomous systems  \citep{ogunsina2024advanced} and the role that PFs play in these systems, it is then  crucial to understand the behaviour of the probability $\P(\sup_{t\in\{1,\dots,T\}}\|\hat{\eta}_t^N-\hat{\eta}_t\|\geq \kappa)$   as a function of the number of particles  $N$ and of the time horizon $T$.

The second contribution of this paper is to provide a first sharp result on the dependence of  the probability $\P(\sup_{t\in\{1,\dots,T\}}\|\hat{\eta}_t^N-\hat{\eta}_t\|\geq \kappa)$ to $N$ and $T$. Considering the  same filtering problem as the one described above, we prove that there exists a  constant $C\in(0,\infty)$ such that, for any $(\kappa,q)\in(0,1)^2$ and  any $T\geq 1$, the inequality $\P(\sup_{t\in\{1,\dots,T\}}\|\hat{\eta}_t^{N}-\hat{\eta}_t\|\geq \kappa)\leq  q$ holds for all $N\geq N_{T,\kappa,q}:=C v_\kappa \log(T/q)$, with $v_\kappa=\{1+\log(1+\kappa^{-1})\}^2 \kappa^{-2}$. We also establish that, for a given $\kappa$, the dependence of $N_{T,\kappa,q}$ on the two parameters $T$ and $q$ is  sharp. 

To the best of our knowledge,  prior to this work the only study of the probability $\P(\sup_{t\in\{1,\dots,T\}}\|\hat{\eta}_t^{N}-\hat{\eta}_t\|\geq \kappa)$ was given in \citet[][Corollary 14.5.7, page 449]{del2013mean}. In this reference, it is shown that the inequality $\P(\sup_{t\in\{1,\dots,T\}}\|\hat{\eta}_t^{N}-\hat{\eta}_t\|\geq \kappa)\leq  q$ holds for all $N\geq N'_{T,\kappa,q}:=C' \kappa^{-2} \log(1+T)(1+\log(1/q))$, with   $C'$ a finite constant. This result is however obtained under     strong  assumptions,  which are for instance not satisfied for the filtering problem  that we consider in this work. We also note that     $N'_{T,\kappa,q}$ has worse dependence on   the parameters $(T,q)$   than    $N_{T,\kappa,q}$, although its dependence on $\kappa$ is slightly better.  Lastly, and importantly,   it has not been realized (and therefore proved) in \citet{del2013mean} that    $N$ \textit{must} increase  with $T$ to ensure that the probability $\P(\sup_{t\in\{1,\dots,T\}}|\hat{\eta}_t^{N}(f)-\hat{\eta}_t(f)|\geq \kappa)$ remains bounded above by $q$ as $T$ increases, and thus that for any fixed $N\geq 1$ we have $\P(\sup_{t\geq 1}\|\hat{\eta}_t^{N}-\hat{\eta}_t\|\geq \kappa)=1$.

Despite  the  reassuring results we proved regarding the time evolution of PF estimates,  it is not entirely satisfactory to delegate (life-critical)  filtering tasks to   algorithms that we know will fail infinitely often. Following a previous discussion, to resolve this limitation of PFs  some de-randomization is  needed, and a possible approach to do so is to replace its underpinning  independent $\mathcal{U}(0,1)$ random variables by randomized quasi-Monte   point sets. This idea has been proposed and studied  in \citet{gerber2015sequential} with the aim of obtaining a filtering algorithm, called sequential quasi-Monte Carlo (SQMC), which has  a faster convergence rate (as $N\rightarrow\infty$) than PFs. Letting $(\tilde{\eta}^N_t)_{t\geq 1}$ denote the estimate of $(\hat{\eta}_t)_{t\geq 1}$ obtained with  SQMC, the third contribution of this paper is to show, for the one-dimensional filtering problem outlined above, that we have  $\lim_{N\rightarrow\infty}\sup_{t\geq 1}\| \tilde{\eta}^N_t-\hat{\eta}_t\|=0$ with probability one.

\subsection{Notation and organization of the paper}

In what follows all the random variables are defined on the probability space $(\Omega,\F,\P)$, and for any  $s\in\mathbb{N}$  we let $\mathcal{B}(\R^s)$ denote  the Borel $\sigma$-algebra on $\R^s$, $\mathcal{P}(\R^s)$ denote the set of probability measures on $(\R^s,\mathcal{B}(\R^s))$, $\lambda_s$ denote   the Lebesgue measure on $\R^s$,   and  for any point set $\{u^n\}_{n=1}^N$ in $[0,1)^s$  we denote by $D^*_s\big(\{u^n\}_{n=1}^N\big)$ its star discrepancy, that is
\begin{align*}
D^*_s\big(\{u^n\}_{n=1}^N\big)=\sup_{b\in [0,1)^s}\Big|\frac{1}{N}\sum_{n=1}^N\ind_{[0,b)}(u^n)-\lambda_s\big([0,b)\big)\Big|.
\end{align*}

Next, for any probability measure  $\pi\in\mathcal{P}(\R)$  we let $F_\pi$ denote   its cumulative density function and  $F^{-1}_\pi$  denote the (generalized-)inverse of $F_\pi$.   In this notion,  the Kolmogorov distance between the probability measures $\pi_1,\pi_2\in\mathcal{P}(\R)$ is given by $\|\pi_1-\pi_2\|=\sup_{a\in\R}|F_{\pi_1}(a)-F_{\pi_2}(a)|$, and we denote by $\|\pi_1-\pi_2\|_\mathrm{D}$ the discrepancy metric between these two probability measures, that is we let
\begin{align*}
\|\pi_1-\pi_2\|_\mathrm{D}=\sup_{-\infty< a<b<\infty}\big|\pi_1([a,b])-\pi_2([a,b])\big|.
\end{align*}
We recall the reader that these two  norms on $\mathcal{P}(\R)$ are   related together through the inequalities
\begin{align}\label{eq:KD}
\|\pi_1-\pi_2\|\leq \|\pi_1-\pi_2\|_\mathrm{D}\leq 2 \|\pi_1-\pi_2\|,\quad\forall  \pi_1,\pi_2\in\mathcal{P}(\R).
\end{align}
With the exception of Propositions \ref{prop:PF} and \ref{prop:lower}, the results presented below are expressed in term of the Kolmogorov metric, which is arguably more popular than the discrepancy metric. These results    can however be easily expressed in term of the latter  metric by using \eqref{eq:KD}.  Propositions \ref{prop:PF} and \ref{prop:lower} are expressed using the discrepancy metric for reasons that will be provided   in due course (see Remarks \ref{rem:prop:PF} and \ref{rem:prop:lower}).

To introduce some further notation we let $\pi\in\mathcal{P}(\R)$,   $f:\R\rightarrow\R$ be   a measurable function  and  $K$ be a    (potentially  un-normalized) kernel acting on $(\R,\mathcal{B}(\R))$. Then, we let $\pi(f)=\int_{\R} f(x)\pi(\dd x)$,   $\pi K$ denote the  measure on $\mathcal{B}(\R)$  such that $\pi K(A)=\int_{\R} K(x,A)\pi(\dd x)$ for all $A\in\mathcal{B}(\R)$, and    $(K^k)_{k\geq 1}$ denote the sequence of (potentially  un-normalized) kernels  acting on $(\R,\mathcal{B}(\R))$ defined by 
\begin{align*}
K^k(x,A)=\int_{\R}K(x,\dd x') K^{k-1}(x',A),\quad (x,A)\in\R\times \mathcal{B}(\R),\quad  k\geq 1.
\end{align*}
Lastly, below we let  $\Phi(\cdot)$ denote   the cumulative density function of the $\mathcal{N}(0,1)$ distribution.

The rest of the paper is organized as follows. In the next subsection we provide a brief reminder on  scrambled  $(t,s)$-sequences. Section \ref{sec:model_Algo} introduces formally the filtering problem and algorithms studied in this work, and our main results are presented in Section \ref{sec:main}. Main results   not proved in Section \ref{sec:main} are proved in Section \ref{sec:proofs}, while Section \ref{app:aux} contains the proof of Proposition \ref{prop:stability_model} presented Section \ref{sec:model_Algo} as well as the proof of Lemma \ref{lemma:key1}   stated Section \ref{sec:proofs}, which is the main technical contribution of the paper.

\subsection{Reminder on scrambled \texorpdfstring{$(t,s)$}{Lg}-sequences\label{sub:digital}}

Let $t\geq 0$, $s\geq 1$ and $b\geq 2$ be three integers, and let
\begin{align*}
\left\{\prod_{j=1}^s\left[a_j b^{-d_j},(a_j+1)b^{-d_j}\right)\subseteq [0,1)^s,\, a_j,\,d_j\in\mathbb{N},\, a_j< b^{d_j},
\, j=1,...,s\right\}
\end{align*}
be the set of all $b$-ary boxes.  Then, for an integer $m\geq t$  the point set $\{u^n\}_{n=1}^{b^m}$ is called a $(t,m,s)$-net in base $b$ if every $b$-ary box of volume $b^{t-m}$ contains exactly $b^t$ points, while a sequence  $(u^n)_{n\geq 1}$ of points in $[0,1)^s$ is  called a $(t,s)$-sequence in base $b$ if, for any integers $a\geq 0$ and $m\geq  t$, the point set $\{u^n\}_{n=ab^m+1}^{(a+1)b^m}$ is a $(t,m,s)$-net in base $b$. 

From the above  definitions it should be clear   that, from a practical perspective,   $(t,s)$-sequences in base $b$ with $t=0$ and $b=2$ are preferable, but such  sequences exist only for   $s\in\{1,2\}$  \citep[][Corollary 4.24, page 62]{Niederreiter1992}. In the context of this work, the crucial property of a $(t,s)$-sequence  $(u^n)_{n\geq 1}$ in base $b$  is   that, for some constant $C_{s,t,b}\in(0,\infty)$ depending only on $s$, $t$ and $b$, we have \citep[see][Section 4.1]{Niederreiter1992}
\begin{align}\label{eq:di}
D^*_s\big(\{u^n\}_{n=1}^N\big)\leq C_{s,t,b}\frac{ \log (N)^s}{N},\quad\forall N\geq b^t.
\end{align}

Finally, a sequence  $(U^n)_{n\geq 1}$ of $(0,1)^s$-valued random variables is a scrambled  $(t,s)$-sequence in base $b$ if, $\P$-a.s.,  $(U^n)_{n\geq 1}$ is a $(t,s)$-sequence in base $b$ and if $U^n\sim\mathcal{U}(0,1)^s$ for all $n\geq 1$. Consequently,  a scrambled  $(t,s)$-sequence in base $b$  is a sequence  $(U^n)_{n\geq 1}$ of dependent $\mathcal{U}(0,1)^s$ random variables such that  $D^*_s\big(\{U^n\}_{n=1}^N\big)\leq  C_{s,t,b} N^{-1}\log(N)^s$ with $\P$-probability one and for all $N\geq 1$, and with the constant $C_{s,t,b}$ as in \eqref{eq:di}. A scrambled  $(t,s)$-sequence is obtained by randomizing, or scrambling, a $(t,s)$-sequence, and various scrambling methods have been proposed in the literature. It is worth mentioning that the first $N$ points of a scrambled sequence can be generated in $\bigO(N\log (N))$ operations,    but the  memory requirements of the algorithm needed to perform this task depend on the scrambling procedure that is used \citep{hong2003algorithm}. We refer the reader to Chapter 3 in \citet{dick2010digital} and to   \citet{practicalqmc}   for a discussion on scrambling algorithms.

\section{Model and filtering algorithms\label{sec:model_Algo}}

\subsection{The Model\label{sub:model}}

We let $(Y_t)_{t\geq 1}$ be a sequence of $\R$-valued random variables and throughout this work we consider the SSM which assumes that, for some constants  $\rho\in\R$ and   $(\sigma,c)\in(0,\infty)^2$, and  some  distribution $\eta_1\in\mathcal{P}(\R)$, the sequence  $(Y_t)_{t\geq 1}$ is such that
\begin{align}\label{eq:SSM}
Y_t=X_t+c^{-1/2}Z_t,\quad X_{t+1}=\rho X_{t}+\sigma W_{t+1},\quad\forall t\geq 1
\end{align}
with $X_1\sim \eta_1$ and where the $Z_t$'s and the $W_{t+1}$'s are independent $\mathcal{N}(0,1)$ random variables. 

We now let $(y_t)_{t\geq 1}$ be a sequence in $\R$ and, for all $t\geq 1$, we denote by $\hat{\eta}_t$  the  filtering distribution of $X_t$, that is the conditional distribution of $X_t$ given that $Y_s=y_s$ for all $s\in\{1,\dots,t\}$ under the model \eqref{eq:SSM}. In addition,  for all $t\geq 2$ we denote by $\eta_t$  the predictive distribution  of $X_t$, that is the conditional distribution of $X_t$ given that $Y_s=y_s$ for all $s\in\{1,\dots,t-1\}$ under the model \eqref{eq:SSM}.

We assume henceforth that $y_t=0$ for all $t\geq 1$ and, for all $z\in\R$, we let $M(z,\dd x)= \mathcal{N}(\rho z,\sigma^2)$, $G(z)=e^{-c z^2/2}$ and $Q(z,\dd x)=M(z,\dd x)G(x)$. With this notation and assumption on $(y_t)_{t\geq 1}$, the two sequences of distributions $(\hat{\eta}_t)_{t\geq 1}$ and $(\eta_t)_{t\geq 1}$ are defined by
\begin{align}\label{eq:seq_1}
\hat{\eta}_1(\dd x_1)=\frac{G(x_1)\eta_1(\dd x_1)}{\eta_1(G)}
\end{align}
and by
\begin{align}\label{eq:seq_2}
\eta_t(\dd x_t)= \hat{\eta}_{t-1}M(\dd x_t),\quad  \hat{\eta}_t(\dd x_t)=\frac{\hat{\eta}_1 Q^{t-1}(\dd x_t)}{\hat{\eta}_1 Q^{t-1}(\R)},\quad t\geq 2.
\end{align}

The following proposition shows that the sequences $(\hat{\eta}_t)_{t\geq 1}$ and $(\eta_t)_{t\geq 1}$ admit  a well-defined limit in $\mathcal{P}(\R)$ and that, as $t\rightarrow\infty$, the impact of the initial distribution $\eta_1$ on $\hat{\eta}_t$ and on $\eta_t$ vanishes exponentially fast.

\begin{proposition}\label{prop:stability_model}
There exist  constants  $(C_\star,\sigma^2_\infty)\in (0,\infty)^2$ and $\epsilon_\star\in(0,1)$, where only $C_\star$  dependents on $\eta_1$, such that   
\begin{align*}
 \|\hat{\eta}_t-\mathcal{N}(0,\sigma^2_\infty)\|\leq C_\star\epsilon_\star^t,\quad\big\|\eta_{t}-\mathcal{N}(0,\rho^2\sigma^2_\infty+\sigma^2)\big \|\leq C_\star\epsilon_\star^t,\quad   \forall t\geq 1.
\end{align*}
\end{proposition}
\begin{remark}
When $\eta_1$ is a Gaussian distribution the conclusion of the lemma can be obtained from the results in \citet{del2023theoretical}.
\end{remark}
 
We stress that, despite the simplicity of the filtering problem that we consider in this work,   little is known about the finite $N$ behaviour of the PF estimates $(\hat{\eta}_t^N )_{t\geq 1}$ and  $(\eta_t^N )_{t\geq 1}$ of  the two sequences $(\hat{\eta}_t)_{t\geq 1}$ and  $(\eta_t )_{t\geq 1}$.  Notably, using Corollary 3 in \citet{caffarel2024mathematical}  only allows to establish  that, for all functions $f$ in some class  of  functions (containing unbounded functions)  and under the assumption that $|\rho|<1$, we have  $\sup_{t\geq 1}\E\big[|\eta_t^N(f)-\eta_t(f)|]\leq C N^{-\beta/2}$   for some unknown constants $C$ and $\beta\in(0,1]$.

\subsection{Particle filters and sequential-quasi Monte Carlo\label{sub:PF}}

In the specific case where $\eta_1=\mathcal{N}(\mu_1,\sigma_1^2)$ for some constants $\mu_1\in\R$ and $\sigma_1^2\in(0,\infty)$,  the   SSM  defined  in  \eqref{eq:SSM} is a linear Gaussian SSM and the two sequences $(\hat{\eta}_t)_{t\geq 1}$ and $(\eta_t)_{t\geq 1}$ are sequences of Gaussian distributions whose means and variances can be computed using the Kalman filter.  For other choices of initial distribution $\eta_1$ these two sequences are intractable and therefore need  to be approximated.
 
Algorithm \ref{algo:PF}  describes a genetic filtering algorithm that can be used to perform this task. When in Algorithm \ref{algo:PF}  the $U_t^n$'s are independent $\mathcal{U}(0,1)^2$ random variables we recover the bootstrap PF with multinomial resampling, which is arguably the simplest PF algorithm.  By contrast, we recover SQMC when, in Algorithm \ref{algo:PF}, for all $t\geq 1$ the set $\{U_t^n\}_{n=1}^N$ is the first $N$ points of a scrambled $(t,2)$-sequence in based $b$ (see Section \ref{sub:digital} for a definition), and these random sets are independent. 


\begin{algorithm}[t]
 
\begin{algorithmic}[1]

\vspace{0.1cm}

\State  let $\{U_1^n=(U_{1,1}^n,U_{1,2}^n)\}_{n=1}^N$ be a random point set in $(0,1)^2$
\vspace{0.1cm}

\State let $X_1^n=F^{-1}_{\eta_1}(U_{1,1}^n)$ and  $W_1^n=G(X_1^n)/\sum_{m=1}^N G(X_1^m)$
\vspace{0.1cm}

\State let $\eta_1^N=\frac{1}{N}\sum_{n=1}^N  \delta_{\{X_1^n\}}$ and  $\hat{\eta}_1^N=\sum_{n=1}^N W_1^n \delta_{\{X_1^n\}}$

\For{$t\geq 2$}
\vspace{0.1cm}

\State  let $\{U_t^n=(U_{t,1}^n,U_{t,2}^n)\}_{n=1}^N$ be a random point set in $(0,1)^2$
\vspace{0.1cm}

\State let  $\hat{X}_{t-1}^n=F^{-1}_{\hat{\eta}_{t-1}^N}(U_{t,1}^n)$ \hfill [resampling] 

\State\label{mut}let  $X_t^n=F^{-1}_{M(\hat{X}_{t-1}^n,\dd x)}(U_{t,2}^n)$ \hfill [mutation] 

\State let $W_t^n=G(X_t^n)/\sum_{m=1}^N G(X_t^m)$
\State let $\eta_t^N=\frac{1}{N}\sum_{n=1}^N  \delta_{\{X_t^n\}}$ and  $\hat{\eta}_t^N=\sum_{n=1}^N W_t^n \delta_{\{X_t^n\}}$ 

\EndFor
\end{algorithmic}
\caption{Generic   filtering algorithm \\ (Operations with index $n$ must be performed for all $n\in \{1,\dots,N\}$)\label{algo:PF}}
\end{algorithm}

\section{Main results\label{sec:main}}
 
\subsection{Almost sure behaviour of particle filters\label{sub:AS}}

The following proposition is the  first result announced in the introductory section (see also Remark \ref{rem:prop:PF}):

\begin{proposition}\label{prop:PF}
Let $N\geq 1$ and consider Algorithm \ref{algo:PF} where    $\big(\{U_{t,2}^n\}_{n=1}^N\big)_{t\geq 1}$ is a sequence of independent sets of $N$   independent $\mathcal{U}(0,1)$ random variables. Then,  
\begin{align}\label{eq:PF_instable}
\P\big(\hat{\eta}_t^N\big([a,b] \big)=0,\,\,i.o.\big)=\P\big( \eta_t^N\big([a,b] \big)=0,\,\,i.o.\big)=1, \quad  \forall [a,b]\subset \R
\end{align}
and we have
\begin{align*}
\P\big(\|\hat{\eta}_t^N-\hat{\eta}_t\|_{\mathrm{D}}\geq \kappa,\,\,i.o.\big)=\P\big(\|\eta_t^N-\eta_t\|_{\mathrm{D}}\geq \kappa,\,\,i.o.\big)=1, \quad\forall\kappa\in (0,1).
\end{align*}
\end{proposition}
\begin{remark}\label{rem:prop:PF}
Using \eqref{eq:KD}, it follows   that if we replace the discrepancy metric $\|\cdot\|_{\mathrm{D}}$ by the Kolmogorov metric $\|\cdot\|$ then the conclusion of Proposition \ref{prop:PF} holds only for all $\kappa\in (0,1/2)$.
\end{remark}

\subsubsection{Proof of Proposition \ref{prop:PF}\label{sub:proof_propPF}}

\begin{proof}[Proof of Proposition \ref{prop:PF}]
To prove the first statement of the proposition  we let $\F^N_0=\{0,\Omega\}$ be the trivial $\sigma$-algebra,
 \begin{align*}
 \F^N_t=\sigma\Big( \big\{(X_s^n,\hat{X}_s^n),\,n=1,\dots,N,\,\, s=1,\dots,t\Big),\quad\forall t\geq 1
\end{align*} 
and $[a,b]\subset\R$  be     arbitrary. Then, letting $\varrho=\min\big\{1-\Phi(2(b-a)/\sigma),1/2\big\}>0$, we easily check  that
\begin{align*}
\inf_{x':\,\rho x'\in[a,b]} M\big(x', [a,b]^c\big)\geq  1-\Phi\big(2(b-a)/\sigma\big)\geq\varrho ,\quad \inf_{x':\,\rho x'\not\in[a,b]} M\big(x', [a,b]^c\big)\geq \frac{1}{2} \geq \varrho
\end{align*}
from which we deduce that
\begin{align}\label{eq:is_PF}
\P\big(\eta_t^N([a,b])=0\big|\F^N_{t-1}\big)\geq \varrho^N>0,\quad\forall t\geq 1,\quad \P-a.s.
\end{align}
By using the second Borel-Cantelli lemma for dependent events given in \citet[][Theorem 4.3.4, page 225]{durrett2019probability}, it follows from \eqref{eq:is_PF}   that with $\P$-probability one we have $\eta_t^N([a,b])=0$ infinitely often.  Since we have the implication $\eta_t^N([a,b])=0\implies \hat{\eta}_t^N([a,b])=0$, the proof of   \eqref{eq:PF_instable} is complete.

To prove the second statement of the proposition we let $\kappa\in (0,1)$ be arbitrary, and  we let $a_\kappa\in(0,\infty)$ and $T_\kappa\in\mathbb{N}$ be such that  $\hat{\eta}_t([-a_\kappa,a_\kappa])\geq \kappa$  for all $t\geq T_\kappa$. Remark that such constants $a_\kappa$ and $T_\kappa$ exist by Proposition \ref{prop:stability_model}. Then, for any $t\geq T_\kappa$ we have the implication 
\begin{align*}
\hat{\eta}_t^N([-a_\kappa,a_\kappa])=0\implies \|\hat{\eta}_t^N-\hat{\eta}_t\|_{\mathrm{D}}\geq |\hat{\eta}_t^N([a_\kappa,a_\kappa])-\hat{\eta}_t([-a_\kappa, a_\kappa])| \geq \kappa
\end{align*}
and thus it follows from \eqref{eq:PF_instable} that    $\P\big(\|\hat{\eta}_t^N-\hat{\eta}_t\|_{\mathrm{D}}\geq \kappa,\,\,i.o.\big)=1$. Following similar calculations  we obtain that $\P\big(\|\eta_t^N-\eta_t\|_{\mathrm{D}}\geq \kappa,\,\,i.o.\big)=1$, completing the prove  of the proposition.
\end{proof}

\subsubsection{Some remarks\label{sub:Disc_PF}}

The result stated in \eqref{eq:PF_instable} is the key result of Proposition \ref{prop:PF},  the second part of the proposition  following directly from \eqref{eq:PF_instable} and    Proposition \ref{prop:stability_model}.

Inspecting the proof of Proposition \ref{prop:PF} reveals that
the negative result  \eqref{eq:PF_instable}  concerning the behaviour of PFs  arises for a simple and intuitive reason. On the one hand, since for all $t\geq 1$ the uniform random variables $\{U_{t,2}^n\}_{n=1}^N$ used by a PF are independent, for any interval $[a,b]\subset\R$ there exists a constant $\varrho>0$ such that, after each mutation step of the algorithm (Line \ref{mut} in Algorithm \ref{algo:PF}), all the particles are outside  the interval $[a,b]$ with  probability at least $\varrho^N$. On the other hand, because the different mutation steps of a PF  rely on  independent sets of uniform random variables, with probability one this event occurs infinity often, by a suitable version of the second Borel-Cantelli lemma.

Following the discussion in the previous paragraph, and more formally following  the calculations in Section \ref{sub:proof_propPF}, it should be clear that the result in \eqref{eq:PF_instable} is not specific to the considered filtering problem but will hold in general. Notably, the problem identified in \eqref{eq:PF_instable} is  entirely caused by the mutation steps of PF algorithms. This observation  implies, for instance, that \eqref{eq:PF_instable}  still holds   if  for all $t\geq 1$ we replace   the function $G(x)=e^{-c x^2/2}$ by any  measurable function $G_t:\R\rightarrow (0,\infty)$. We also stress that \eqref{eq:PF_instable}  is not due to the fact that we consider an unbounded   state-space.  To illustrate this point  let $\setX=[-l,l]$ for some  constant $l\in (0,\infty)$, and let $(\hat{\eta}_t)_{t\geq 1}$ and $(\eta_t)_{t\geq 1}$ be as defined in \eqref{eq:seq_1}-\eqref{eq:seq_2} where, for all $x'\in\R$, $M(x',\dd x)$ denote the $\mathcal{N}(\rho x',\sigma^2)$ distribution truncated on $\setX$ and where $\eta_1$ is such that $\eta_1(\setX)=1$. In this context, by repeating the   calculations in Section \ref{sub:proof_propPF}, it is trivial to verify that if in Algorithm \ref{algo:PF}   the   $U_{t,2}^n$'s are independent $\mathcal{U}(0,1)$ random variables then
\begin{align*}
\P\big(\hat{\eta}_t^N\big( [-a,a] \big)=0,\,\,i.o.\big)=\P\big(\eta_t^N\big [-a,a]\big)=0,\,\,i.o.\big)=1,\quad\forall N\geq 1,\quad\forall a\in (0,l).
\end{align*}

Interestingly, in the literature it has been realized that if, in Algorithm \ref{algo:PF}, the function $G$ is defined by $G(x)=\ind_{\R\setminus I}(x)$ for some  compact interval $I\subset \R$ then  it may happen that, at some time $t\geq 1$,   all the particles $\{X_t^n\}_{n=1}^N$ fall  into $I$. In this case, the particle system is said to be extinct, and some estimates for  the  probability that this event happens   before a given time instant $T\geq 2$ are provided in \citet[Section 7.4.1, page 231]{DelMoral:book}.  Surprisingly, it has not been noticed that the problem whereby all the particles  $\{X_t^n\}_{n=1}^N$ are located in an arbitrary given compact interval $I\subset\R$   (i) arises independently of the definition of the function $G$,  and (ii)  implies that PFs cannot be used to reliably estimate filtering distributions over an infinite time horizon.

 \subsection{Finite time horizon and probabilistic behaviour of particle filters}
 
By Proposition \ref{prop:PF}, as time progresses a PF will inevitably reach a time instant $\tau$ at which it will fail  at estimating well the filtering and/or the predictive distribution  of the SSM of interest. From a practical point of view it is of key   interest to study the probability that such an undesirable event happens in a finite number $T$ of time steps.

The following theorem addresses this question:
\begin{theorem}\label{thm:PF}
Let $N\geq 1$ and consider Algorithm \ref{algo:PF} where    $\big(\{U_{t}^n\}_{n=1}^N\big)_{t\geq 1}$ is a sequence of independent sets of $N$  independent $\mathcal{U}(0,1)^2$ random variables. Then, there exists a constant $\bar{C}_1\in (0,\infty)$ (independent of $N$) such that, for all $q\in(0,1)$ and all $T\geq 1$, with probability at least $1-q$ we have  
\begin{align*}
\max\bigg\{\sup_{t\in\{1,\dots,T\}} \|\hat{\eta}_t^N-\hat{\eta}_t\|,\,
\sup_{t\in\{1,\dots,T\}} \|\eta_t^N-\eta_t\|\bigg\}\leq    \bar{C}_1        \delta_{N,T,q}^{ 1/2}\log\big(1+\delta_{N,T,q}^{-1/2}\big) 
\end{align*}
with    $\delta_{N,T,q}=N^{-1 }\big(1+\log(T/q)\big)$.
\end{theorem}

The next result is a   direct consequence of  Theorem \ref{thm:PF}.
\begin{corollary}\label{cor:PF}
Consider the set-up of Theorem \ref{thm:PF}. Then,  there exists a constant $\bar{C} '_1\in(0,\infty)$ such that, for any $(q,\kappa)\in(0,1)^2$ and $T\geq 1$, we have
\begin{align*}
\P\Big(\sup_{t\in\{1,\dots,T\}} \|\hat{\eta}_t^N-\hat{\eta}_t\|\leq \kappa,\,\,\sup_{t\in\{1,\dots,T\}} \|\hat{\eta}_t^N-\hat{\eta}_t\|\leq \kappa\Big)\geq 1-q,\quad\forall N\geq N_{T,\kappa,q}
\end{align*}
with $N_{T,\kappa,q}=\bar{C}'_1\big\{1+ \log(1+\kappa^{-1})\big\}^2\kappa^{-2}\big(1+  \log(T/q)\big)$.
\end{corollary}

The lower bound on $N$ provided in Corollary \ref{cor:PF}  to ensure a given (probabilistic) estimation error depends on $(T,q)$ only through the quantity $\log(T/q)$. As shown in the next proposition, this dependence in the parameters $T$ and $q$ is sharp:
\begin{proposition}\label{prop:lower}
Consider the set-up of Theorem \ref{thm:PF} and let $\bar{q}\in (0,1/2)$. Then, for all $\kappa\in(0,1]$, there exist  constants  $C_\kappa\in(0,\infty)$ and $T_\kappa\in\mathbb{N}$  such that, for all $q\in(0,\bar{q}]$ and all $T\geq T_\kappa$, we have
\begin{align*}
\P\Big(\sup_{t\in\{1,\dots,T\}}\|\hat{\eta}_t^{N}-\hat{\eta}_t\|_{\mathrm{D}}\leq \kappa,\,\,\sup_{t\in\{1,\dots,T\}}\|\eta_t^{N}-\eta_t\|_{\mathrm{D}}\leq \kappa\Big)< 1-q,\quad\forall N<  C_\kappa\log(T/q).
\end{align*}
\end{proposition} 
\begin{remark}\label{rem:prop:lower}
Using \eqref{eq:KD}, it follows   that if we replace the discrepancy metric $\|\cdot\|_{\mathrm{D}}$ by the Kolmogorov metric $\|\cdot\|$ then the conclusion of Proposition \ref{prop:PF} holds only for all $\kappa\in (0,1/2)$.
\end{remark}

\subsubsection{Proof of  Proposition \ref{prop:lower}\label{sub:prop:lower}}
\begin{proof}[Proof of Proposition \ref{prop:lower}]
Let $\kappa\in(0,1)$  be arbitrary and remark that, by  Proposition \ref{prop:stability_model} and \eqref{eq:KD},  there exist  constants    $a_\kappa\in(0,\infty)$  and $T_\kappa\in\mathbb{N}$   such that, for all $t\geq T_\kappa$, we have both $\eta_t([-a_\kappa,a_\kappa])> \kappa   $ and $\hat{\eta}_t([-a_\kappa,a_\kappa])>\kappa$. In addition, let $\varrho_\kappa=\min\big\{1-\Phi(4a_\kappa/\sigma),1/2\big\}$ and, for all $N\geq 1$, let $(\F^N_t)_{t\geq 0}$ be as defined in the proof of Proposition \ref{prop:PF}. Then, by using \eqref{eq:is_PF},  we have
\begin{align}\label{eq:Ex_N}
\P\Big(\hat{\eta}_t^N([-a_\kappa,a_\kappa])=0,\,\, \eta_t^N([-a_\kappa,a_\kappa])=0|\F^N_{t-1}\Big)\geq  \varrho^N_\kappa,\quad\forall t\geq 1,\quad\forall N\geq 1.
\end{align}
Therefore, noting that for all $t\geq T_\kappa$  and $N\geq 1$ we have
\begin{align*}
1-\P\Big(\|\hat{\eta}_t^{N}-\hat{\eta}_t\|_{\mathrm{D}}\leq \kappa,&\,\,\|\eta_t^{N}-\eta_t\|_{\mathrm{D}}\leq \kappa|\F_{t-1}^N\Big)\geq \P\Big(\hat{\eta}_t^N([-a_\kappa,a_\kappa])=0,\,\,\eta_t^N([-a_\kappa,a_\kappa])=0|\F_{t-1}^N\Big) 
\end{align*}
it follows from \eqref{eq:Ex_N} that, for all $T\geq T_\kappa$ and $N\geq 1$,
\begin{align}\label{eq:Ex_N2}
\P\Big(\sup_{t\in\{1,\dots,T\}}\|\hat{\eta}_t^{N}-\hat{\eta}_t\|_{\mathrm{D}}\leq \kappa,\,\,\sup_{t\in\{1,\dots,T\}}\|\eta_t^{N}-\eta_t\|_{\mathrm{D}}\leq \kappa\Big)\leq (1-\varrho^N_\kappa)^{T-T_\kappa+1}.
\end{align}
To conclude the proof of the proposition we let $T\geq T_\kappa$, $\bar{q}\in (0,1/2)$ and $q\in(0,\bar{q}]$ be fixed, and to simplify the notation we let $v_T=T-T_\kappa+1$. Then, we first note that
\begin{align}\label{eq:Ex_N3}
(1-\varrho^N_\kappa)^{v_T}< 1-q,\quad\forall N< N_{T,\kappa,q}:=\frac{-\log\big( 1-(1-q)^{1/v_T} \big)}{\log(1/\varrho_\kappa)} 
\end{align}  
where, using Bernoulli's inequality \citep[see e.g.][Theorem 3.1]{li2013some}, 
\begin{align*}
(1-q)^{1/v_T}=\frac{(1-q)^{1+1/v_T}}{1-q}\geq \frac{1-q-q/v_T}{1-q}.
\end{align*}
The latter inequality implies that  $1-(1-q)^{1/v_T}\leq     (q/v_T)/(1-q)\leq  (q/v_T)/(1-\bar{q})$,  and thus
 \begin{align}\label{eq:Ex_N4}
 N_{T,\kappa,q}\geq \frac{\log(v_T/q)+\log(1-\bar{q})}{\log(1/\varrho_\kappa)}\geq \frac{\log(v_T/q)}{\log(1/\varrho_\kappa)}\delta_{\bar{q}},\quad\delta_{\bar{q}}:=1+\frac{\log(1-\bar{q})}{\log(1/\bar{q})}
 \end{align}
where $\delta_{\bar{q}}>0$  since $\bar{q}\in(0,1/2)$. The result of the proposition  follows from \eqref{eq:Ex_N2}-\eqref{eq:Ex_N4}, upon noting that if $T_\kappa>1$ then
 \begin{align*}
 \frac{\log(v_T/q)}{\log(T/q)}\geq \frac{\log(v_T)}{\log(T)}\geq \inf_{x\in[T_\kappa,\infty)}\frac{\log( x-T_\kappa+1)}{\log(x)}>0.
 \end{align*}
 where the first inequality uses the fact that $T\geq v_T$.
 
\end{proof}

 \subsubsection{Some remarks}

Corollary \ref{cor:PF}   constitutes a  reassuring result  concerning the time evolution of PF estimates. Indeed, it  notably shows that, as the time horizon $T$ that we consider increases,  the number of particles $N$ only needs to    grow   at speed $\log(T)$ to ensure that   the  PF algorithm    estimates  the whole sets of distributions $\{\hat{\eta}_t\}_{t=1}^T$ and $\{\eta_t\}_{t=1}^T$ with a constant (probabilistic) error.

It is important to stress that the result of  Proposition \ref{prop:lower} is not specific to the considered filtering problem but will hold in general. Indeed, by inspecting the calculations of Section \ref{sub:prop:lower}, it should be clear that the conclusion of Proposition \ref{prop:lower} holds as soon as the result of Proposition \ref{prop:stability_model} and   \eqref{eq:PF_instable}  hold. The former result is standard  while,  as we argued in Section \ref{sub:Disc_PF},  the latter result is expected to hold for  a large class of filtering problems.

\subsection{Almost sure behaviour of  sequential quasi-Monte Carlo}

We now turn to the study of SQMC. In the context of this work we only need  point sets in $(0,1)^2$ and thus, following the  discussion in Section \ref{sub:digital}, we limit our attention to scrambled $(0,2)$-sequences in base $b=2$. More precisely, in this subsection we study  Algorithm \ref{algo:PF} in the case where $\{U_t^n\}_{n=1}^N$ is the first $N$ points of a scrambled $(0,2)$-sequence in base $b=2$ for all $t\geq 1$, in which case we have, by \citet[][Theorem 4.7, page 54 and Theorem 4.14, page 59]{Niederreiter1992},
\begin{align}\label{eq:D_bounds}
\sup_{t\geq 1}D^*_2\big(\{U_t^n\}_{n=1}^N\big)\leq  \delta_N,\quad\forall N\geq 1\quad\P-a.s.
\end{align}
with
\begin{align}\label{eq:bound_02}
\delta_N=
\begin{cases}
\frac{\log(N)+3}{2N}, &\text{ if }N\in\{2^k,\,k\in\mathbb{N}\}\\
\frac{ \log (N)^2+11\log(2)\log(N)+18\log (2)^2}{N 8\log (2)^2}, &\text{otherwise} 
\end{cases},\quad\forall N\geq 1.
\end{align}

For the   resulting   filtering algorithm we have the following almost sure and time uniform guarantee:

\begin{theorem}\label{thm:main}
Let $N\geq 1$ and consider Algorithm \ref{algo:PF} where $\{U_t^n\}_{n=1}^N$ is the first $N$ points of a scrambled $(0,2)$-sequence in base $b=2$ for all $t\geq 1$. Then, there exists a constant $\bar{C}_2\in (0,\infty)$ (independent of $N$) such that, with $\delta_N$ as defined in \eqref{eq:bound_02},
\begin{align*}
\max\bigg\{\sup_{t\geq 1} \|\hat{\eta}_t^N-\hat{\eta}_t\|,\,\,\sup_{t\geq 1} \|\eta_t^N-\eta_t\|\bigg\}\leq   \bar{C}_2\,\delta_N^{1/2 }\log\big(1+\delta_N^{-1/2}\big),\quad\P-a.s.
\end{align*}
\end{theorem}
 
\begin{remark}
Theorem \ref{thm:main} does not assume that the random point sets $\{U_t^n\}_{n=1}^N$'s are independent.  
\end{remark}

 Theorem \ref{thm:main} implies that the SQMC estimates $(\hat{\eta}_t^N)_{t\geq 1}$ and $(\eta_t^N)_{t\geq 1}$ of the two sequences $(\hat{\eta}_t)_{t\geq 1}$ and $(\eta_t)_{t\geq 1}$ are such that $\lim_{N\rightarrow\infty}\sup_{t\geq 1} \|\hat{\eta}_t^N-\hat{\eta}_t\|=\sup_{t\geq 1} \| \eta_t^N-\eta_t\|=0$ with $\P$-probability one. Moreover, it allows to readily obtain the following result:
\begin{corollary}\label{cor:SQMC}
Consider the set-up of Theorem \ref{thm:main}. Then, there exists a constant $\bar{C} '_2\in(0,\infty)$ such that, for any $(q,\kappa)\in(0,1)^2$ and $T\geq 1$, the following implication holds:
\begin{align*}
\frac{N}{\log(1+N)^2}\geq \bar{C}_2'\big\{1+& \log(1+\kappa^{-1})\big\}^2\kappa^{-2}\implies\P\Big(\sup_{t\geq 1} \|\hat{\eta}_t^N-\hat{\eta}_t\|\leq \kappa,\,\,\sup_{t\geq 1} \|\hat{\eta}_t^N-\hat{\eta}_t\|\leq \kappa\Big)=1.
\end{align*}
\end{corollary}

\subsubsection{Some remarks}
 
To simplify the discussion we focus below on the estimation of the filtering distributions but the same comments hold for the estimation  of the predictive distributions $(\eta_t)_{t\geq 1}$.
 
Due to the presence of the $\log(N)$ term in the definition of $\delta_N$,  the dependence in $N$ of the bound given in   Theorem \ref{thm:main} for $\sup_{t\geq 1} \|\hat{\eta}_t^N-\hat{\eta}_t\|$ is slightly worse than that of the probabilistic bound obtained in Theorem \ref{thm:PF} for PFs. On the other hand,   in the context of this work, SQMC has the advantage    to be a   filtering algorithm having time uniform and almost sure guarantees. 

Unfortunately, we cannot compare the lower bound on $N$ given in Corollary \ref{cor:SQMC} with the one obtained  in Corollary \ref{cor:PF}  to ensure that,  with probability at least $1-q$,  the PF estimates $\{\hat{\eta}_t^N\}_{t=1}^T$ are such  that $\sup_{t\in\{1,\dots,T\}} \|\hat{\eta}_t^N-\hat{\eta}_t\|\leq \kappa$: there is indeed   no reason  for the  constants $\bar{C}_1$ and $\bar{C}_2$ appearing in these bounds  to be identical. Given the properties of scrambled $(t,s)$-sequences  reminded in Section \ref{sub:digital}, it is however reasonable to expect that, in the worst scenario, the constant $\bar{C}_2$ appearing in  Corollary \ref{cor:SQMC} can only be slightly larger than the constant $\bar{C}_1$ appearing in  Corollary \ref{cor:PF}.

 \section{Proof  of Theorems \ref{thm:PF} and \ref{thm:main} and  of Corollaries \ref{cor:PF} and \ref{cor:SQMC}\label{sec:proofs}}

 \subsection{Preliminary result: A general discrepancy bound for filtering algorithms}

Let $\Phi,\Psi:\mathcal{P}(\R)\rightarrow\mathcal{P}(\R)$ be defined 
\begin{align*}
\Psi(\mu)(\dd x)=\frac{G(x)\mu(\dd x)}{\mu(G)} ,\quad\Phi(\mu)=\Psi(\mu)M,\quad \mu\in\mathcal{P}(\R)
\end{align*} 
and note  that $\eta_{t+1}=\Phi(\eta_t)$ for all $t\geq 1$.

Using this notation,  the following result plays a central role in the proof of  Theorems \ref{thm:PF} and \ref{thm:main}:
\begin{lemma}\label{lemma:key1}
There exists   a constant  $C\in (0,\infty)$ such that, for any sequence $(\mu_t)_{t\geq 1}$   in $\mathcal{P}(\R)$ and letting $\hat{\mu}_t=\Psi(\mu_t)$ for all $t\geq 1$,  we have, for any $T\geq 1$,
\begin{align*}
 \sup_{t\in\{1,\dots,T\}} \|\hat{\mu}_t-\hat{\eta}_t\|\leq   C\, \Delta_{T} \log\big(1+\Delta_{T}^{-1}\big),\quad \sup_{t\in\{1,\dots,T\}} \|\mu_t-\eta_t\|\leq C\, \Delta_{T} \log\big(1+\Delta_{T}^{-1}\big)
\end{align*}
 where    $\Delta_{T}=\sup_{t\in\{1,\dots,T\}}\|\mu_t-\Phi(\mu_{t-1})\|$  with the convention    that $\Phi(\mu_{0})=\eta_1$.
\end{lemma}
\begin{proof}
See Section \ref{sub-p-lemma:key1}.
\end{proof}

\begin{remark}
Let $(\mu_t)_{t\geq 1}$ be a sequence in $\mathcal{P}(\R)$ such that, for all $t\geq 1$, we have $\mu_{t+1}=\Phi_t(\mu_t)$ for some mapping $\Phi_t:\mathcal{P}(\R)\rightarrow\mathcal{P}(\R)$. Then, informally speaking, Lemma \ref{lemma:key1} implies that if $\Phi_t$ is close to $\Phi$ for all $t\in\{1,\dots,T\}$ then the estimation error $\|\mu_t-\eta_t\|$ will be small for all $t\in\{1,\dots,T\}$. This message conveyed by Lemma \ref{lemma:key1} is very intuitive, but, to the best of our knowledge, has not been proved to be true before.
\end{remark}

\subsection{Proof of Theorem \ref{thm:PF}}

\begin{proof}

In what follows we let $T\geq 1$ be fixed and we use the convention that $\hat{\eta}^N_0M=\eta_1$. Then, by  Lemma  \ref{lemma:key1}, to obtain the result of Theorem \ref{thm:PF} it suffices to control the behaviour of  $\sup_{t\in\{1\,\dots,T\}}\|\eta_t^N-\hat{\eta}^N_{t-1} M\|$ as a function of $N$.

To this aim,  remark first that, with $\P$-probability one and for all $t\geq 1$ and all $N\geq 1$, the function $F_{\hat{\eta}^N_{t-1} M}:\R\rightarrow(0,1)$ is continuous. Therefore, by the  Dvoretzky-Kiefer-Wolfowitz inequality \citep[see][Corollary 1]{massart1990tight},   for all $ (\gamma,\kappa)\in (0,1)$ we have
\begin{align}\label{eq:Massart}
\sup_{t\in\{1,\dots,T\}}\P\Big(\|\eta_t^N-\hat{\eta}^N_{t-1} M\|\leq \kappa \big|\F_{t-1}^N\Big)\geq 1-\gamma,\quad\forall N\geq N_{\kappa,\gamma}:=\frac{\log(2/\gamma)}{2\kappa^2}
\end{align}
with the $\sigma$-algebra $\F_t^N$ as defined in the proof of Proposition \ref{prop:PF} for all $t\geq 0$ and all $N\geq 1$. Using \eqref{eq:Massart}, we can conclude that
\begin{align*}
\P\Big(\sup_{t\in\{1,\dots,T\}}\|\eta_t^N-\hat{\eta}^N_{t-1} M\|\leq \kappa \Big)\geq (1-\gamma)^T,\quad\forall N\geq N_{\kappa,\gamma},\quad\forall (\gamma,\kappa)\in (0,1).
\end{align*} 
We now let $q\in (0,1)$. Then, by applying the latter result with $\gamma=1-(1-q)^{1/T}$, we obtain  that
\begin{align}\label{eq:Massart3}
\P\Big(\sup_{t\in\{1,\dots,T\}}\|\eta_t^N-\hat{\eta}^N_{t-1} M\|\leq \kappa \Big)\geq 1-q,\quad\forall N\geq  N_{\kappa,1-(1-q)^{1/T}}. 
\end{align}
By Bernoulli's inequality \citep[see e.g.][]{li2013some},  we have $1-(1-q)^{1/T}\geq  T^{-1}q$  and thus
\begin{align*}
N_{\kappa,1-(1-q)^{1/T}}\leq  \frac{\log(2)+\log(T/q)}{2\kappa^2}\leq  \frac{\log(2)+\log(T/q)}{\kappa^2}.
\end{align*}
Together with \eqref{eq:Massart3}, this shows that
\begin{align*}
\P\Big(\sup_{t\in\{1,\dots,T\}}\|\eta_t^N-\hat{\eta}_{t-1} M\|\leq N^{-1/2}\sqrt{ \log(2)+\log(T/q)}\Big)\geq 1-q 
\end{align*}
and the result of Theorem \ref{algo:PF}  follows from  Lemma  \ref{lemma:key1}.
\end{proof}

\subsection{Proof of Theorem \ref{thm:main}}

\subsubsection{Preliminary result: A discrepancy bound for sampling from univariate mixtures}

The proof of the following lemma is inspired by that of \citet[][``Statz 2'']{Hlawka1972}.

\begin{lemma}\label{lemma:QMC_sample}
Let $x_1,\dots,x_M$ be $M$ points in $\R$ for some $M\geq 1$, and   for all $x\in\{x_1,\dots,x_M\}$   let $\mu_x\in\mathcal{P}(\R)$. In addition, let $\pi_M(\dd x)=\sum_{m=1}^M W_m \delta_{\{x_m\}}\in\mathcal{P}(\R)$ and $\mu=\sum_{m=1}^M (W_m \delta_{\{x_m\}}\otimes\mu_{x_m}) \in \mathcal{P}(\R^2)$. Lastly, let  $\varphi_\mu: (0,1)^2\rightarrow\R^2$ be defined by
\begin{align*}
\varphi_\mu(u)=\Big(z_1, F^{-1}_{\mu_{x_1}}(u_2)\Big),\quad  z_1=F^{-1}_{\pi_M}(u_1),\quad u=(u_1,u_2)\in(0,1)^2  
\end{align*}
and let  $\{u^n\}_{n=1}^N$ be a point set in $(0,1)^2$ for some $N\geq 1$. Assume that, for all $a\in\R$, the mapping $ \{x_1,\dots,x_M\}\ni x\mapsto F_{\mu_x}(a)$ is monotone and let $\mathcal{B}_2=\big\{\prod_{i=1}^2 (-\infty,a_i],\, (a_1,a_2)\in \R^2 \big\}$. Then, 
\begin{align}\label{eq:bound_QMCsample}
\sup_{B\in\mathcal{B}_2}\Big|\frac{1}{N}\sum_{n=1}^N\ind_B\big(\varphi_\mu(u^n)\big)-\mu(B)\Big|\leq   \sqrt{96} \,D^*_2\big(\{u^n\}_{n=1}^N\big)^{1/2}+12 D^*_2\big(\{u^n\}_{n=1}^N\big).
\end{align}
\end{lemma}
\begin{proof}

 Let $B=(-\infty, a_1]\times (-\infty,a_2]$ for some $(a_1,a_2)\in\R^2$. Without loss of generality we assume below that the set $\{x_m\}_{m=1}^M$ is labelled in such a way that $x_1\leq x_2\leq\dots\leq x_M$. In addition, we assume that $a_1\geq x_1$ since  otherwise  the set $B$ is such that 
\begin{align*}
\Big|\frac{1}{N}\sum_{n=1}^N\ind_B\big(\varphi_\mu(u^n)\big)-\mu(B)\Big|=0.
\end{align*}
To prove the lemma we let $G(B)=\big\{ u\in (0,1)^2: \varphi_\mu(u)\in B\big\}$ so that
\begin{align}\label{eq:step1}
\frac{1}{N}\sum_{n=1}^N\ind_B\big(\varphi_\mu(u^n)\big)-\mu(B)=\frac{1}{N}\sum_{n=1}^N\ind_{G(B)}(u^n)-\lambda_2(G(B)).
\end{align}
Letting $j_B=\max\{i\in\{1,\dots,M\}: x_i\leq a_1\}$, and using the convention that $W_{0}=0$, we note that
\begin{align}\label{eq:G_set}
G(B)=\bigcup_{i=1}^{j_B}\Big(W_{i-1}, W_i\Big]\times \big(0, F_{\mu_{x_i}}(b_2)\big]\subseteq (0,1]^2.
\end{align}

To proceed further we let $L\in\mathbb{N}$ and  $\mathsf{P}$ be the partition of $(0,1]^2$ in congruent squares of side $\frac{1}{L}$  such that $W\in \mathsf{P}$ if and only if $W= (\xi_1/L, (\xi_1+1)/L]\times  (\xi_2/L, (\xi_2+1)/L]$ for some $(\xi_1,\xi_2)\in\{0,\dots,L-1\}^2$. Then, we let
\begin{align*}
\mathsf{U}_1=\Big\{ W\in \mathsf{P}:\,\, W\subset\mathring{G(B)}\Big\},\quad  \mathsf{U}_2=\Big\{ W\in \mathsf{P}:\,\, W\cap\partial( G(B))\neq\emptyset\Big\}
\end{align*}
and  we let $U_1=\cup_{W\in \mathsf{U}_1}W$ and $U_2=\cup_{W\in \mathsf{U}_2}W$ with the convention that $\cup_{W\in \mathsf{S}}=\emptyset$ if $\mathsf{S}= \emptyset $. With this notation in place, and letting $U_1'=G(B)\setminus U_1$, we have
\begin{equation}\label{eq:split_dis}
\begin{split}
\frac{1}{N}\sum_{n=1}^N\ind_{G(B)}(u^n)-\lambda_2(G(B))&=\frac{1}{N}\sum_{n=1}^N\ind_{U_1}(u^n)-\lambda_2(U_1)+\frac{1}{N}\sum_{n=1}^N\ind_{U_1'}(u^n)-\lambda_2(U_1') 
\end{split}
\end{equation}
and we start by studying the first term on the r.h.s.~of \eqref{eq:split_dis}.

To do so assume first that    $\mathsf{U}_1\neq\emptyset$ and, for  all $\xi_1\in\{0,\dots,L-1\}$, let
\begin{align*}
S_{\xi_1}=\Big\{k\in\{0,\dots,L-1\}:\,\, (\xi_1/L, (\xi_1+1)/L]\times  (k/L, (k+1)/L]\in \mathsf{U}_1\Big\}
\end{align*}
and let  $I=\{\xi_1\in\{0,\dots,L-1\}:\, S_{\xi_1}\neq\emptyset\}$. Then, using the expression \eqref{eq:G_set} for $G(B)$, it is direct to see that for all $\xi_1\in I$ the set
\begin{align*}
Q(\xi_1):=\bigcup_{k\in S_{\xi_1}} (\xi_1/L, (\xi_1+1)/L]\times  (k/L, (k+1)/L]
\end{align*}
is such that $Q(\xi_1)=(a,b]\subset(0,1]^2$ for some $a,b\in [0,1]^2$. Therefore, noting that $\cup_{\xi_1\in I} Q(\xi_1)=U_1$, we obtain that 
\begin{equation}\label{eq:split_dis_p1}
\begin{split}
\Big|\frac{1}{N}\sum_{n=1}^N\ind_{U_1}(u^n)-\lambda_2(U_1)\Big|  \leq \sum_{\xi_1\in I}\Big|\frac{1}{N}\sum_{n=1}^N\ind_{Q(\xi_1)}(u^n)-\lambda_2(Q(\xi_1))\Big|&\leq 4|I|  D^*_2\big(\{u^n\}_{n=1}^N\big)\\
&\leq 4L\,  D^*_2\big(\{u^n\}_{n=1}^N\big)
\end{split}
\end{equation}
where the second inequality uses the fact that, since $\{u^n\}_{n=1}^N$ is assumed to be a point set in $(0,1)^2$, we have \citep[see][Remarks 2-3 and Proposition 2.4, pages 14-15]{Niederreiter1992}
\begin{align}\label{eq:dis_01}
\sup_{(a,b]\subset (0,1]^2}\Big|\frac{1}{N}\sum_{n=1}^N\ind_{(a,b]}(u^n)-\lambda_2\big((a,b]\big)\Big|\leq 4 D^*_2\big(\{u^n\}_{n=1}^N\big).
\end{align}
If   $\mathsf{U}_1=\emptyset$ then $\frac{1}{N}\sum_{n=1}^N\ind_{U_1}(u^n)-\lambda_2(U_1)=0$ and thus \eqref{eq:split_dis_p1} also holds in this case.

We now study the second term on the r.h.s.~of \eqref{eq:split_dis}. To this end we note first that we can cover $U_1'$ with sets in $\mathsf{U}_2$, and thus 
\begin{align*}
-\lambda_2(U_2)\leq  -\lambda_2(U_1')\leq \frac{1}{N}\sum_{n=1}^N\ind_{U_1'}(u^n)-\lambda_2(U_1')\leq \frac{1}{N}\sum_{n=1}^N\ind_{U_2}(u^n)-\lambda_2(U_2)+\lambda_2(U_2).  
\end{align*}
By using this latter result, we obtain that 
\begin{equation}\label{eq:split_dis_p2}
\begin{split}
\Big|\frac{1}{N}\sum_{n=1}^N\ind_{U_1'}(u^n)-\lambda_2(U_1')\Big|&\leq \Big|\frac{1}{N}\sum_{n=1}^N\ind_{U_2}(u^n)-\lambda_2(U_2)\Big|+\lambda_2(U_2)\leq|\mathsf{U}_2|\Big(4D^*_2\big(\{u^n\}_{n=1}^N\big)+L^{-2}\Big)
\end{split}
\end{equation}
where the second   inequality holds by \eqref{eq:dis_01}.

Using \eqref{eq:G_set} and the fact that, by assumption, the function $\{x_1,\dots,x_M\}\ni x\mapsto F_{\mu_x}(b_2)$ is either non-increasing or non-decreasing,  it is readily checked that $|\mathsf{U}_2|\leq 2L$. Together with  \eqref{eq:step1} and \eqref{eq:split_dis}-\eqref{eq:split_dis_p2}, this implies that
\begin{align}\label{eq:lastLem}
\sup_{B\in\mathcal{B}_2}\Big|\frac{1}{N}\sum_{n=1}^N\ind_B\big(\varphi_\mu(u^n)\big)-\mu(B)\Big|\leq    12L D^*_2\big(\{u^n\}_{n=1}^N\big)+2L^{-1} 
\end{align}
and the result of the lemma follows  by applying \eqref{eq:lastLem} with 
$L=\big\lceil  \big\{6 D^*_2\big(\{u^n\}_{n=1}^N\big)\big\}^{-1/2}\big\rceil$.
\end{proof}

\subsubsection{Proof of the theorem}

\begin{proof}[Proof of Theorem \ref{thm:main}]
By Lemma \ref{lemma:QMC_sample},   for all $N\geq 1$ and using the convention that $\hat{\eta}^N_0M=\eta_1$, we have $\sup_{t\geq 1}\|\eta_t^N-\hat{\eta}_{t-1}^N M\|\leq \sqrt{238} D^*_2\big(\{U_{t}^n\}_{n=1}^N\big)$  and the result of Theorem \ref{thm:main} then follows from   Lemma  \ref{lemma:key1} and \eqref{eq:D_bounds}.
 \end{proof}

\subsection{Proof of Corollary \ref{cor:PF}}
\begin{proof}
In what follows we   let $(\kappa,q)\in(0,1)$ be arbitrary and $\delta_{N,T,q}=N^{-1 }\big(1+\log(T/q)\big)$ for all integers $N\geq 1$ and $T\geq 1$. In addition, we let $C_\kappa\in(\exp(2),\infty)$ so that the inequality $ \delta_{N,T,q}^{-1/2}\geq C_\kappa^{1/2}\kappa^{-1}\geq\exp(1)$ holds for all $N \geq C_\kappa \kappa^{-2}\big(1+  \log(T/q)\big)$. Then, using the fact that function $f(x)=x^{-1}\log(1+x)$ is non-increasing on $x\in(\exp(1),\infty)$, it follows that 
\begin{align*}
N&\geq C_\kappa \kappa^{-2}\big(1+  \log(T/q)\big)\implies   \delta_{N,T,q}^{ 1/2}\log\big(1+\delta_{N,T,q}^{-1/2}\big)\leq C_\kappa^{-1/2}\kappa\log\big(1+ C_\kappa^{1/2}\kappa^{-1}\big)
\end{align*}
where, with the constant $\bar{C}_1\in(0,1)$ as in the statement of Theorem \ref{thm:PF},
\begin{align}\label{eq:C_constraint}
\bar{C}_1C_\kappa^{-1/2}\kappa\log\big(1+ C_\kappa^{1/2}\kappa^{-1}\big)\leq \kappa\Leftrightarrow  \frac{\log\big(1+C_\kappa^{1/2}\kappa^{-1}\big)}{C_\kappa^{1/2}}\leq \frac{1}{\bar{C}_1}.
\end{align}
To proceed further we let  $D\in(1,\infty)$ and $f_D:(0,\infty)\rightarrow\R$ be defined by
\begin{align*}
f_D(x)=\frac{\log\big(1+x D v(x)\big)}{D v(x)},\quad v(x):=\exp(1)+\log(1+x),\quad x\in(0,\infty).
\end{align*}
Simple calculations show that the function $x\mapsto \log(1+x)/\log(x)$ is non-increasing  on the interval $(1,\infty)$.  Therefore, noting that, since $D\geq 1$, we have $xD v(x)\geq \exp(1)>1$ for all $x\geq 1$, it follows that 
\begin{equation}\label{eq:fD1}
\begin{split}
f_D(x)=\frac{\log\big(x Dv(x)\big)}{D v(x)}\frac{\log\big(1+x Dv(x)\big)}{\log\big(x Dv(x\big))}\leq  \frac{\log\big(x Dv(x)\big)}{D v(x)} \log\big(1+\exp(1)\big),\quad\forall x\geq 1
\end{split}
\end{equation}
where
\begin{align}\label{eq:fD2}
\frac{\log\big(x D v(x)\big)}{D v(x)}\leq \frac{1}{D}\bigg(1+  \frac{\log(D)}{\exp(1)}+\sup_{z\geq 1}\frac{\log\big(v(x))}{v(x)}\bigg),\quad\forall x\geq 1.
\end{align}
Consequently, by  \eqref{eq:fD1}-\eqref{eq:fD2},  for any $\epsilon\in(0,\infty)$ there exists a constant $D_\epsilon\in(1,\infty)$ such that we have  $\sup_{x\in[1,\infty]}f_{D}(x)\leq \epsilon$ for all $D\in[ D_\epsilon,\infty)$.

By using this latter result, and assuming without loss of generality that $\bar{C}_1>1$, it follows that for any   constant $D\geq D_{1/\bar{C}_1}\geq 1$   the second inequality in \eqref{eq:C_constraint} holds for  $C_\kappa=D^2\big(\exp(1)+\log(1+\kappa^{-1})\big)^2>\exp(2)$. The result of the corollary then follows from Theorem \ref{thm:PF}. 
\end{proof}
 
 \subsection{Proof of Corollary \ref{cor:SQMC}}

 \begin{proof}
The result of Corollary \ref{cor:SQMC} easily follows from Theorem \ref{thm:main} and the calculations done in the proof of Corollary \ref{cor:PF}. The details are therefore omitted to save space.
 \end{proof}

 \section{Proof of Proposition \ref{prop:stability_model} and proof of Lemma  \ref{lemma:key1}  \label{app:aux}}

 \subsection{Additional notation}
 
For a  function $f\in\mathcal{C}^1(\R)$ we let $V(f)=\int_\R |f'(x)|\dd x$ with the convention that $V(f)=0$ if $ f'(x) =0$ for all $x\in\R$. Abusing notation, for a function $f:\R\rightarrow\R$ continuously differentiable on some non-empty interval $(a,b)$ we let $V(\ind_{(a,b]}f)=V(\ind_{(a,b)}f)=\int_a^b |f'(x)|\dd x$. Next, for all $(z,r)\in\R^2$ and $s\in (0,\infty)$   we let $M_{r,s^2}(z,\dd x)=\mathcal{N}(r z,s^2)$ and $G_s(z)=e^{-s z^2/2}$, and below $\varphi(\cdot)$  denote the probability  density function of the $\mathcal{N}(0,1)$ distribution. Moreover, for all $(a,r)\in\R^2$ and $s\in(0,\infty)$ we let $h_{a,r,s}:\R\rightarrow [-1,1]$ be defined by
\begin{align}\label{h_def}
h_{a,r,s}(x)=\Phi\bigg(\frac{a-r x}{s}\bigg)-\Phi\bigg(\frac{a }{s}\bigg),\quad x\in\R 
\end{align}
while, for any    bounded function $H:\R\rightarrow (0,\infty)$,  we let $\Psi_H:\mathcal{P}(\R)\rightarrow \mathcal{P}(\R)$ be defined by
\begin{align*}
\Psi_H(\pi)=\frac{H(x)\pi(\dd x)}{\pi(H)},\quad \pi\in\mathcal{P}(\R).
\end{align*}

\subsection{Preliminary results}

\subsubsection{A simple but useful technical lemma}

\begin{lemma}\label{lemma:h_function}
For all $(a,r)\in\R^2$ and $s\in(0,\infty)$ we have $V(h_{a,r,s})\leq 1$ and
\begin{align*}
| h_{a,r,s}(x)|\leq   \frac{|rx|}{s\sqrt{2\pi}},\quad\forall x\in\R.
\end{align*}
\end{lemma}
\begin{proof}
Let $(a,r)\in\R^2$ and $s\in(0,\infty)$. If $r=0$ the result of the lemma is trivial and we therefore assume  that $r\neq 0$ in what follows.   In this scenario, we have
\begin{align*}
V(h_{a,r,s})=\int_\R|h'_{a,r,s}(x)|\dd x&=  \int_{\R}\frac{|r|}{s} \varphi\Big(\frac{a-  rx}{s}\bigg)\dd x= \int_{\R}\frac{1}{s} \varphi\Big(\frac{a- z}{s}\Big)  \dd z=1 
\end{align*}
showing the first part of the lemma. On the other hand, using the mean value theorem and noting that $|\Phi'(x)|=|\varphi(x)|\leq 1/\sqrt{2\pi}$ for all $x\in\R$, we obtain that
\begin{align*}
| h_{a,r,s}(x)|&=\bigg|\Phi\bigg(\frac{a-r x}{s}\bigg)-\Phi\bigg(\frac{a}{s}\bigg)\bigg|\leq  \frac{|rx|}{s\sqrt{2\pi}},\quad\forall x\in\R 
\end{align*}
showing the second part of the lemma. The proof is complete.
\end{proof}

\subsubsection{A general 1-dimensional Koksma-Hlawka inequality}

The next result is a general $s=1$ dimensional Koksma-Hlawka inequality, where the integration is w.r.t.~an arbitrary distribution $\mu\in\mathcal{P}(\R)$. By contrast, in the original Koksma-Hlawka inequality, the dimension $s\geq 1$ is arbitrary but the integration is w.r.t.~the $\mathcal{U}(0,1)^s$ distribution  \citep[see e.g.][Proposition 2.18, page 33]{dick2010digital}. It is also worth mentioning that  \citet{aistleitner2014functions} derived  a version of the Koksma-Hlawka inequality where the integration is w.r.t.~an arbitrary probability distribution on $(0,1)^s$ for some arbitrary $s\geq 1$.

\begin{lemma}\label{lemma:KH_in}
Let $-\infty\leq a<b\leq \infty$ and let $f:\R\rightarrow\R$ be continuously differentiable on the interval $(a,b)$ and such that $V(f\ind_{(a,b)})<\infty$,    and let    $\pi,\mu\in\mathcal{P}(\R)$. Then,   with the convention that $f(a)=0$ when $a=-\infty$ and with the convention that $f(b)=0$ when $b=\infty$, we have
\begin{equation*}
\begin{split}
\Big|\int_{a}^b f(x)\big(\pi(\dd x)-\mu(\dd x)\big)\Big|\leq \Big(|f(a)|+|f(b)|+V( f\ind_{(a,b)})\Big)\|\pi-\mu\|.
\end{split}
\end{equation*}
\end{lemma}
\begin{proof}

Assume first $(a,b)\in\R^2$ and note that,  since the function $f$ is continuous, we have $\int_{a}^b f(x)\dd\big(F_\pi(x)-F_\mu(x)\big)= \int_{a}^b  f(x)(\pi(\dd x)-\mu(\dd x))$. Then,   using the integration by part formula for   Riemann-Stieltjes integrals \citep[see][Theorem 7.6, page 144]{apostol}, we have
\begin{equation*}
\begin{split}
\int_{a}^b \big(F_\pi(x)-F_\mu(x)\big) f'(x)\dd x&=\int_{a}^b \big(F_\pi(x)-F_\mu(x)\big) \dd f (x)\\
&=f(b)\big(F_\pi(b)-F_\mu(b)\big)-f(a)\big(F_\pi(a)-F_\mu(a)\big)-\int_{a}^b f(x)\dd\big(F_\pi(x)-F_\mu(x)\big)
\end{split}
\end{equation*}
from which we deduce that
\begin{equation}\label{eq:Rs}
\begin{split}
\Big|\int_{a}^b  f(x)(\pi(\dd x)-\mu(\dd x))\Big|&\leq |f(b)|\, \big|F_\pi(b)-F_\mu(b)\big|+|f(a)|\,\big|F_\pi(a)-F_\mu(a)\big|\\
&+\Big|\int_{a}^b \big(F_\pi(x)-F_\mu(x)\big) f'(x)\dd x\Big|.
\end{split}
\end{equation}
The result of the lemma when $(a,b)\in\R^2$ directly follows from \eqref{eq:Rs}.

Assume now that $a=-\infty$ and that $b\in\R$. Note that under the assumptions of the lemma we have $\sup_{x\in\R}|f(x)|<\infty$ while, by using the reverse  Fatou's lemma, we obtain that $\lim_{x\rightarrow -\infty}F_\pi(x)=\lim_{a\rightarrow-\infty}F_\mu(x)=0$. Therefore, using \eqref{eq:Rs},
\begin{align*}
\Big|\int_{-\infty}^b  f(x)(\pi(\dd x)-\mu(\dd x))\Big|&=\Big|\lim_{a\rightarrow-\infty}\int_{a}^b f(x)\dd\big(F_\pi(x)-F_\mu(x)\big)\Big|\\
&\leq \big|f(b)|\,\big|F_\pi(b)-F_\mu(b)|\\
&+\limsup_{a\rightarrow-\infty}\Big|\int_{a}^b \big(F_\pi(x)-F_\mu(x)\big) f'(x)\dd x\Big|\\
&\leq \|\pi-\mu\|\Big(|f(b)|+ V(\ind_{(-\infty,b)} f)\Big)
\end{align*}
proving the result of the lemma  when $a=-\infty$ and $b\in\R$.

Using Fatou's lemma we have $\lim_{x\rightarrow\infty}F_\pi(x)=\lim_{a\rightarrow\infty}F_\mu(x)=1$, and    the result of the lemma for the two remaining case (i.e.~for $b=\infty$ and $a\in\R\cup\{-\infty\}$) can proved using similar calculations as above, and is  therefore omitted to save space.
\end{proof}

\subsubsection{A useful property of the model of Section \ref{sub:model}\label{sub:lemmaQ}}

\begin{lemma}\label{lemma:Q}
Let  $c_0=\sigma_0=0$, $\rho_0=1$, and let $(c_k)_{k\geq 1}$, $(\rho_k)_{k\geq 1}$ and $(\sigma_k)_{k\geq 1}$ be such that
\begin{align*}
c_{k}= \frac{(c+c_{k-1})\rho^2}{1+(c+c_{k-1})\sigma^2},\quad \rho_{k}=\prod_{i=0}^{k-1}\frac{\rho}{1+(c+c_i)\sigma^2},\quad \sigma_{k}^2=\sigma_{k-1}^2+\frac{\rho_{k-1}^2\sigma^2}{1+(c+c_{k-1})\sigma^2}
\end{align*}
for all $k\geq 1$. Then, for all $k\geq 1$, we have
 \begin{align}\label{eq:Qk}
Q^k(x',\dd x)\propto G_{c_k}(x')M_{\rho_k,\sigma^2_k}(x',\dd x),\quad\forall x'\in\R.
\end{align}
In addition, the sequence $(c_k)_{k\geq 0}$ is such that $c\rho^2/(1+c\sigma^2)\leq c_k\leq \rho^2/\sigma^2$ for all $k\geq 1$, and there exist constants $(\sigma_\infty,C_\star)\in(0,\infty)^2$ and $\epsilon_\star\in(0,1)$ such that $|\rho_k|\leq C_\star \epsilon_\star^k$ and such that $|\sigma_k^2-\sigma_\infty^2|\leq C^2_\star\epsilon_\star^{2k}$ for all $k\geq 0$.
\end{lemma}

\begin{proof}
The result stated in \eqref{eq:Qk}   can be deduced from the calculations in \citet{DelMoral2001} while the convergence result for the sequence $(\sigma_k)_{k\geq 1}$  can be directly obtained from \citet[][Theorem 3.5]{del2023theoretical}. However, for sake of completeness, a proof of the whole lemma is presented below.

To prove the first part of the lemma  remark first that, for all $a\in (0,\infty)$,  we have
\begin{align}\label{eq:Qk_1}
M_{\rho,\sigma^2}(x',\dd x)G_a(x)\propto G_{\tilde{a}}(x')M_{\tilde{\rho},\tilde{\sigma}^2}(x',\dd x),\quad\forall x'\in\R
\end{align}
where
\begin{align}\label{eq:Qk_2}
\tilde{a}=\frac{a\rho^2}{1+a\sigma^2},\quad  \tilde{\rho}=\frac{\rho}{1+a \sigma^2},\quad \tilde{\sigma}^2=\frac{\sigma^2}{1+a\sigma^2}.
\end{align}
We now  prove \eqref{eq:Qk} by induction on $k$. For $k=1$ the result given in \eqref{eq:Qk} directly follows  by applying  \eqref{eq:Qk_1}-\eqref{eq:Qk_2} with $a=c$. Assume now that \eqref{eq:Qk}  holds for some $k\geq 1$. Then, by using first the inductive hypothesis and then by  applying  \eqref{eq:Qk_1}-\eqref{eq:Qk_2} with $a=c+c_k$, for all $x''\in\R$ we have
\begin{align*}
M_{\rho,\sigma^2}(x'',\dd x')G_c(x')Q^k(x',\dd x)&\propto M_{\rho,\sigma^2}(x'',\dd x')G_c(x')G_{c_k}(x')M_{\rho_k,\sigma^2_k}(x',\dd x)\\
&=M_{\rho,\sigma^2}(x'',\dd x')G_{c+c_k}(x')M_{\rho_k,\sigma^2_k}(x',\dd x)\\
&\propto G_{c_{k+1}}(x'') M_{\tilde{\rho}_k,\tilde{\sigma}_k^2}(x'',\dd x') M_{\rho_k,\sigma^2_k}(x',\dd x)
\end{align*}
with
\begin{align*}
\tilde{\rho}_k=\frac{\rho}{1+(c+c_k) \sigma^2},\quad \tilde{\sigma}_k^2=\frac{\sigma^2}{1+(c+c_k)\sigma^2}.
\end{align*}
Therefore, for all $x''\in\R$ we have
\begin{align*}
Q^{k+1}(x'',\dd x)&\propto G_{c_{k+1}}(x'') \int_\R M_{\tilde{\rho}_k,\tilde{\sigma}_k^2}(x'',\dd x') M_{\rho_k,\sigma^2_k}(x',\dd x)\\
&=G_{c_{k+1}}(x'') M_{\tilde{\rho}_k\rho_k, \rho_k^2\tilde{\sigma}^2_k+\sigma^2_k}(x'',\dd x)\\
&=G_{c_{k+1}}(x'') M_{\rho_{k+1}, \sigma^2_{k+1}}(x'',\dd x) 
\end{align*}
and the proof of \eqref{eq:Qk} is complete.

To prove the second part of the lemma let $f:[0,\infty)\rightarrow [0,\infty)$ be defined by
\begin{align*}
f(x)=\frac{(c+x)\rho^2}{1+(c+x)\sigma^2},\quad x\in [0,\infty)
\end{align*}
so that $c_{k}=f(c_{k-1})$ for all $k\geq 1$. Easy computations show that the function $f$ is non-decreasing on  $[0,\infty)$, and thus the sequence $(c_k)_{k\geq 0}$ is non-decreasing and such that $c\rho^2/(1+c\sigma^2)\leq c_k\leq \rho^2/\sigma^2$ for all $k\geq 1$. 

To study the sequence $(\rho_k)_{k\geq 0}$   assume first that $|\rho|<1+c\sigma^2$. In this case, we trivially have that $|\rho_k|\leq \epsilon_\star^k$ for all $k\geq 0$ with $\epsilon_\star=|\rho|/(1+c\sigma^2)$. Assume now that $|\rho|\geq 1+c\sigma^2$ and note that, for all $k\geq 1$, we have
\begin{align}\label{eq:rho_cond}
 \frac{|\rho|}{1+(c+c_k)\sigma^2}<1&\Leftrightarrow  c_k>\frac{|\rho|-(1+c\sigma^2)}{\sigma^2}.
 \end{align}
To proceed further we remark that, since the sequence  $(c_k)_{k\geq 0}$ is non-decreasing and bounded, the sequence converges in $[0,\infty)$ and  we let $c_\infty=\lim_{k\rightarrow\infty}c_k$. In addition, simple calculations show that
\begin{align*}
c_\star:=\frac{\rho^2-(1+c\sigma^2 )+\sqrt{(1+c\sigma^2-\rho^2)^2+4c\sigma^2\rho^2}}{2\sigma^2}
\end{align*}
is the unique value of $x\in [0,\infty)$ such that $x=f(x)$. Moreover, since we are assuming that  $|\rho|\geq 1+c\sigma^2>1$, it follows that $\rho^2-(1+c\sigma^2)\geq 0$ and thus the constant $c_\star$ is such that
\begin{align}\label{eq:c_s_bound}
c_\star\geq \frac{\rho^2-(1+c\sigma^2 )+\sqrt{(1+c\sigma^2-\rho^2)^2 }}{2\sigma^2}=\frac{\rho^2-(1+c\sigma^2 )}{\sigma^2}>\frac{|\rho|-(1+c\sigma^2 )}{\sigma^2}.
\end{align}
On the other hand, since the function $f$ is continuous on $[0,\infty)$ we have 
\begin{align*}
c_\infty=\lim_{k\rightarrow\infty} c_{k+1}=\lim_{k\rightarrow\infty} f(c_{k})=f(\lim_{k\rightarrow\infty} c_{k})=f(c_\infty)
\end{align*}
showing that   $c_\infty=c_\star$, and thus $\lim_{k\rightarrow\infty}c_k=c_\star$. Together with \eqref{eq:c_s_bound}, this implies that there exists a $k_\star\in \mathbb{N}$  such that $c_{k_\star}>(|\rho|-(1+c\sigma^2))/\sigma^2$, and thus, using \eqref{eq:rho_cond} and the fact that the sequence $(c_k)_{k\geq 0}$ is non-decreasing, we have
\begin{align*}
\frac{|\rho|}{1+(c+c_k)\sigma^2}\leq \epsilon_\star:= \frac{|\rho|}{1+(c+c_{k_\star})\sigma^2}<1,\quad\forall k\geq k_\star 
\end{align*}
implying that $|\rho_k|\leq C \epsilon_\star^k$ for all $k\geq 0$ and with $C=\big(|\rho|/(1+c\sigma^2)\big)^{k_\star+1}$. 

Finally, to study the sequence $(\sigma_k)_{k\geq 0}$ note first that
\begin{align*}
\sigma_{k}^2&=\sum_{s=0}^{k-1}\frac{\rho_{s}^2\sigma^2}{1+(c+c_{s})\sigma^2},\quad\forall k\geq 1
\end{align*}
and thus, for all integers $1\leq k<m$ we have, with $C$ and $\epsilon_\star$ as above,
\begin{equation}\label{eq:sig}
\begin{split}
|\sigma_{m}^2-\sigma_{k}^2|=\Big|\sum_{s=k}^{m-1}\frac{\rho_{s}^2\sigma^2}{1+(c+c_{s})\sigma^2}\Big|\leq \frac{\sigma^2}{1+c\sigma^2}\sum_{s=k}^{m-1}\rho_s^2&\leq \frac{\sigma^2 C^2}{1+c\sigma^2}\sum_{s=k}^{m-1}\epsilon_\star^{2s}\leq \epsilon_\star^{2k}\frac{\sigma^2 C^2}{(1-\epsilon_\star)(1+c\sigma^2)}.
\end{split}
\end{equation}
Using \eqref{eq:sig} it is easily verified that   the sequence $(\sigma^2_k)_{k\geq 1}$ is Cauchy and therefore converges in $[0,\infty)$, and we let $\sigma^2_\infty=\lim_{k\rightarrow\infty}\sum_{s=0}^{k-1}\frac{\rho_{s}^2\sigma^2}{1+(c+c_{s})\sigma^2}$. Remark that since the sequence $(\sigma^2_k)_{k\geq 1}$ is non-decreasing and such that $\sigma^2_k\geq \sigma^2/(1+c\sigma^2)$ it follows that $\sigma^2_\infty>0$, and remark that, by \eqref{eq:sig}, we have
\begin{align*}
|\sigma^2_\infty-\sigma_k^2|\leq \epsilon_\star^{2k}\frac{\sigma^2 C^2}{(1-\epsilon_\star)(1+c\sigma^2)},\quad\forall k\geq 1.
\end{align*}
The proof of the lemma is complete.
\end{proof}

\subsection{Proof of Proposition \ref{prop:stability_model}\label{proof-prop:stability_model}}
\begin{proof}

The result of the proposition trivially holds if $\rho=0$ and therefore below we assume that $\rho\neq 0$.

Let $(c_t)_{t\geq 0}$, $(\rho_t)_{t\geq 0}$, $(\sigma_t)_{t\geq 0}$, $C_\star\in(0,\infty)$ and $\epsilon_\star\in(0,1)$ be as  in Lemma \ref{lemma:Q}, and for all $a\in\R$  let $f_a=\ind_{(-\infty,a]}$ and, for all $r\in\R$ and $s\in (0,\infty)$,  let $h_{a,r,s}$ be as defined in \eqref{h_def} and   $p_{a,s}=\Phi(a/s)$. 

We now let $a\in\R$ be fixed. Then, using \eqref{eq:seq_2} and  Lemma \ref{lemma:Q}, we have
\begin{align}\label{eq:h1}
\hat{\eta}_{t+1}(f_a)=\frac{\hat{\eta}_1 Q^{t}(f_a)}{\hat{\eta}_1 Q^{t}(\R)}=\Psi_{G_{c_t}}(\hat{\eta}_1)\big( M_{\rho_{t},\sigma_t^2}(f_a)\big)=\Psi_{G_{c_t}}(\hat{\eta}_1)( h_{a,\rho_{t},\sigma_t})+p_{a,\sigma_t},\quad\forall t\geq 1.
\end{align}
Using Lemma \ref{lemma:h_function} and the fact that $\inf_{t\geq 1}\sigma_t\geq \sigma_1>0$, it follows that  
\begin{align*}
| h_{a,\rho_{t},\sigma_t}(x) |& \leq  \frac{|x \rho_t|}{\sigma_1\sqrt{2\pi}},\quad\forall x\in\R,\quad\forall t\geq 1.
\end{align*}
By combining this latter result and \eqref{eq:h1},  and by using the fact that $\sup_{x\in\R}|x|G_{z}(x)\leq z^{-1/2} e^{-1/2}$ for all $z\in(0,\infty)$, it follows that, for all $t\geq 1$
\begin{equation}\label{eq:qp}
\begin{split}
\big|\hat{\eta}_{t+1}(f_a)-p_{a,\sigma_t}\big|\leq \big| \Psi_{G_{c_t}}(\hat{\eta}_1)( h_{a,\rho_{t},\sigma_t})\big|\leq  \frac{ |\rho_t|e^{-1/2}}{\hat{\eta}_1(G_{c_t})\sigma_1\sqrt{2\pi c_t}}
\leq  \frac{ |\rho_t| \sqrt{1+c\sigma^2}}{\hat{\eta}_1(G_{c_t})\sigma_1\sqrt{  c\rho^2 }} 
\end{split}
\end{equation}
where the last inequality  uses Lemma \ref{lemma:Q}. By   Lemma \ref{lemma:Q}  we have   $G_{c_t}(x)\geq G_{\rho^2/\sigma^2}(x)$ for all $x\in\R$ while $|\rho^t|\leq C_\star \epsilon_\star^t$ for all $t\geq 1$, and thus, by \eqref{eq:qp},
\begin{align}\label{eq:ppp1}
\big|\hat{\eta}_{t+1}(f_a)-p_{a,\sigma_t}\big|\leq C_1|\rho_t|\leq C_1 \epsilon_\star^t,\quad\forall t\geq 1
\end{align}
with 
\begin{align*}
C_1=\frac{C_\star  \sqrt{1+c\sigma^2}}{\hat{\eta}_1(G_{\rho^2/\sigma^2})\sigma_1\sqrt{  c\rho^2 }}.
\end{align*}

On the other hand, using  the mean value theorem, for all $t\geq 1$ there exists a constant  $v_{t,a}\in [-1,1] $ such that
\begin{align*} 
|p_{a,\sigma_t}-p_{a,\sigma_\infty}|&\leq   \frac{|\sigma_t-\sigma_\infty|}{ \sigma_t+v_{t,a}(\sigma_t-\sigma_\infty) }\,\, \frac{|a|}{\sigma_t+v_{t,a}(\sigma_t-\sigma_\infty)}\varphi\Big(\frac{a}{\sigma_t+v_{t,a}(\sigma_t-\sigma_\infty)}\Big).
\end{align*}
Together with the fact that $|x|\varphi(x)\leq 1$ for all $x\in\R$, and recalling $\inf_{t\geq 1}\sigma_t\geq \sigma_1>0$, this shows that
\begin{align}\label{eq:ppp2}
|p_{a,\sigma_t}-p_{a,\sigma_\infty}|\leq    \frac{|\sigma_t-\sigma_\infty|}{\sigma_1},\quad\forall t\geq 1.
\end{align}
By    Lemma \ref{lemma:Q}, we have  $|\sigma^2_t-\sigma^2_\infty| \leq C^2_\star\epsilon_\star^{2t}$ for all $t\geq 1$ and thus, using the fact that $xy\leq (x^2+y^2)/2$ for all $(x,y)\in\R^2$, we have
\begin{align}\label{eq:ppp3}
|\sigma_t-\sigma_\infty|^2\leq  |\sigma^2_t-\sigma^2_\infty| \leq C^2_\star\epsilon_\star^{2t},\quad\forall t\geq 1 .
\end{align}
Noting that in the above calculations the constant $a\in\R$ is arbitrary, by combining \eqref{eq:ppp1}, \eqref{eq:ppp2} and \eqref{eq:ppp3}  we obtain that
\begin{align*}
\|\hat{\eta}_{t+1}-\mathcal{N}(0,\sigma^2_\infty)\|=\sup_{a\in\R}\big|\hat{\eta}_{t+1}(f_a)-p_{a,\sigma_\infty}\big|\leq \big(C_1+C_\star/\sigma_1\big) \epsilon_\star^t,\quad\forall t\geq 1.
\end{align*}

To complete the proof of the proposition it remains to show that
\begin{align}\label{eq:ppp4}
\|\eta_{t+1}-\mathcal{N}(0,\rho^2\sigma^2_\infty+\sigma^2)\|\leq \|\hat{\eta}_t-\mathcal{N}(0,\sigma^2_\infty)\|,\quad\forall t\geq 1.
\end{align}
To this aim we let $\hat{\eta}_\infty=\mathcal{N}(0,\sigma^2_\infty)$ and note that, for all $t\geq 1$, we have
\begin{align}\label{eq:pen}
\|\eta_{t+1}-\mathcal{N}(0,\rho^2\sigma^2_\infty+\sigma^2)\|&=\|\hat{\eta}_{t}M-\hat{\eta}_\infty M\|=\sup_{a\in\R}|\hat{\eta}_{t}(h_{a,\rho,\sigma})-\hat{\eta}_\infty(h_{a,\rho,\sigma})|.
\end{align}
Since $\sup_{a\in\R}V(h_{a,\rho,\sigma})\leq 1$ by Lemma \ref{lemma:h_function}, it follows from Lemma \ref{lemma:KH_in} that
\begin{align*}
\sup_{a\in\R}|\hat{\eta}_{t}(h_{a,\rho,\sigma})-\hat{\eta}_\infty(h_{a,\rho,\sigma})|\leq \|\hat{\eta}_{t}-\hat{\eta}_\infty\|
\end{align*}
and \eqref{eq:ppp4} then follows from \eqref{eq:pen}. The proof of the proposition is complete.
\end{proof}

\subsection{Proof of Lemma \ref{lemma:key1}\label{sub-p-lemma:key1}}

\subsubsection{Two preliminary results}

From Lemma \ref{lemma:KH_in} we readily obtain the following discrepancy bound for 1-dimensional importance sampling:
\begin{corollary}\label{cor:IS}
Let $H\in\mathcal{C}^1(\R)$ be such that $V(H)<\infty$ and such that $H(x)\geq 0$ for all $x\in\R$, and let $\mu,\pi\in\mathcal{P}(\R)$. Then,
\begin{align*}
\|\Psi_H(\mu)-\Psi_H(\pi)\|\leq\frac{ \|H\|_\infty+2V(H) }{\pi(H)}\|\mu-\pi\|.
\end{align*}
\end{corollary}
\begin{proof}
The result follows from Lemma \ref{lemma:KH_in} and the fact that, for all interval $B=(-\infty,a)\subset\R$, we have
\begin{align*}
|\Psi_{H}(\mu)(B)-\Psi_{H}(\pi)(B)|\leq \frac{|(\mu-\pi)(H\ind_B)|+|(\mu-\pi)(H)|}{\pi(H)}.
\end{align*}
\end{proof}

The next lemma is the key intermediate result used to prove Lemma \ref{lemma:key1}:

\begin{lemma}\label{lemma:key0}

Let $I\subseteq\R$ be a compact set. Then, there exists  a constant $C_I\in (1,\infty)$  such that, for any sequence $(\mu_t)_{t\geq 1}$ and letting $\hat{\mu}_t=\Psi_G(\mu_t)$ for all $t\geq 1$, for any $T\geq 1$ we have
\begin{align*}
\ \max\Big\{\sup_{t\in\{1,\dots,T\}} \|\hat{\mu}_t-\hat{\eta}_t\|, \sup_{t\in\{1,\dots,T\}} \|\mu_t-\eta_t\|\Big\}\leq \frac{C_I}{\iota_{T-1}} \Delta_{T} \log\big(1+\Delta_{T}^{-1}\big)
\end{align*}
 where  $\iota_{T-1}=\inf_{t\in\{1,\dots,T-1\}} \mu_t(I)$ and  $\Delta_{T}=\sup_{t\in\{1,\dots,T\}}\|\mu_t-\hat{\mu}_{t-1}M)\|$, with the convention   that $\iota_0=1$  and that $\hat{\mu}_{0}M=\eta_1$.

\end{lemma}

\begin{proof}
 Let $(\mu_t)_{t\geq 1}$ and $(\hat{\mu}_t)_{t\geq 1}$ be as in the statement of the lemma,  $(c_t)_{t\geq 0}$, $(\rho_t)_{t\geq 0}$ and $(\sigma_t)_{t\geq 0}$ be as defined in Lemma \ref{lemma:Q}, and  for all $a\in\R$  let $f_a=\ind_{(-\infty,a]}$ and 
\begin{align*}
g_{a,k}(x)=\Phi\Big(\frac{a- \rho_{k} x}{\sigma_{k}}\Big),\quad h_{a,k}(x)=\Phi\Big(\frac{a- \rho_{k} x}{\sigma_{k}}\Big)-\Phi\Big(\frac{a}{\sigma_{k}}\Big),\quad\forall x\in\R,\quad\forall k\geq 1.
\end{align*}
Next, we let $\phi:\mathcal{P}(\R)\rightarrow  \mathcal{P}(\R)$ be such that $\phi(\eta)=\Psi_G(\eta M)$ for all $\eta\in\mathcal{P}(\R)$, and for all $t\geq 1$  we let $\phi^t=\phi^{t-1}\phi$ with the convention that $\phi^{0}(\eta)=\eta$ for all $\eta\in\mathcal{P}(\R)$.  Note that $\hat{\eta}_{t+1}=\phi^t(\hat{\eta}_1)$ for all $t\geq 1$.

 Using this notation, and following the approach  introduced by  \citet[][Chapter 7]{DelMoral:book} for studying PFs,  we have,   using the convention  that  $\phi(\hat{\mu}_0)=\hat{\eta}_1$,
\begin{align}\label{eq:decomp1}
 \hat{\mu}_t-\hat{\eta}_t &=\sum_{p=1}^t\Big(\phi^{t-p}(\hat{\mu}_p)-\phi^{t-p}\big(\phi (\hat{\mu}_{p-1})\big)\Big),\quad\forall t\geq 1 
\end{align}
and we now study the  terms inside the sums.
 
To this aim, we assume for now    that $T\geq 3$ , and  we let $t\in\{2,\dots,T\}$ and $p\in\{1,\dots,t-1\}$ be fixed. Then, for all $a\in\R$ we have, by Lemma \ref{lemma:Q},
\begin{equation}\label{eq:decomp2}
\begin{split}
\phi^{t-p}(\hat{\mu}_p)(f_a)&=\frac{\hat{\mu}_p Q^{t-p}(f_a)}{\hat{\mu}_p Q^{t-p}(\R)}=\Psi_{G_{c+c_{t-p}}}(  \mu_{p})(g_{a,t-p})
\end{split}
\end{equation}
and 
\begin{align}\label{eq:decomp3}
\phi^{t-p}\big(\phi(\hat{\mu}_{p-1}))(f_a)=\Psi_{G_{c+c_{t-p}}}\big(\hat{\mu}_{p-1}M\big)(g_{a,t-p}).
\end{align}
By combining \eqref{eq:decomp2} and \eqref{eq:decomp3}, we obtain that 
\begin{equation}\label{eq:key_term}
\begin{split}
\sup_{a\in\R}\big|\phi^{t-p}(\hat{\mu}_p)(f_a)-&\phi^{t-p} \big( \phi_{p}(\hat{\mu}_{p-1})\big)(f_a)\big| \\
&= \sup_{a\in\R}\Big|\Psi_{G_{c+c_{t-p}}}(  \mu_{p})(g_{a,t-p})-\Psi_{G_{c+c_{t-p}}}\big(\hat{\mu}_{p-1}M\big)(g_{a,t-p})\Big|\\
&=\sup_{a\in\R}\Big|\Psi_{G_{c+c_{t-p}}}(  \mu_{p})(h_{a,t-p})-\Psi_{G_{c+c_{t-p}}}\big(\hat{\mu}_{p-1}M\big)(h_{a,t-p})\Big|.
\end{split}
\end{equation}
To proceed further we note  that, by Lemma \ref{lemma:h_function}, we have $\sup_{(a,k)\in \R\times\mathbb{N}} V(h_{a,k})\leq 1$ and thus,  by Lemma \ref{lemma:KH_in},
\begin{equation}\label{eq:p2}
\begin{split}
 \sup_{a\in\R}\Big|\Psi_{G_{c+c_{t-p}}} (\mu_{p} )(h_{a,t-p})  -\Psi_{G_{c+c_{t-p}}}&\big( \hat{\mu}_{p-1}M\big)(h_{a,t-p})\Big|\leq \Big\|\Psi_{G_{c+c_{t-p}}} (\mu_{p})-\Psi_{G_{c+c_{t-p}}}\big(\hat{\mu}_{p-1}M\big)\Big\|.
\end{split}
\end{equation}
Noting that
\begin{align}\label{eq:G_var}
\|G_a\|_\infty=1,\quad V(G_a)=\int_\R |G'_a(x)|\dd x= a\int_\R |x| G_a(x)\dd x=2,\quad\forall a\in (0,\infty)
\end{align}
it follows from Corollary \ref{cor:IS}   that
\begin{equation}\label{eq:G_var2}
\begin{split}
\big\|\Psi_{G_{c+c_{t-p}}}(\mu_{p})-\Psi_{G_{c+c_{t-p}}}\big(\hat{\mu}_{p-1}M \big)\Big\| \leq  \frac{5\|  \mu_{p}-\hat{\mu}_{p-1}M\|}{\mu_{p}(G_{c+c_{t-p}})}&\leq  \frac{5\| \mu_{p}-\hat{\mu}_{p-1}M\|}{ \mu_p(G_{c+c_{\star}})}\leq\frac{5 \Delta_{T}}{\mu_p(G_{c+c_{\star}})}
\end{split}
\end{equation}
where the second inequality uses the fact that, by Lemma \ref{lemma:Q}, we have  $\sup_{k\geq 1} c_k\leq c_\star:=\rho^2/\sigma^2$. Noting that since $p<T$ we have $\mu_p(G_{c+c_{\star}})\geq \iota'\iota_{T-1}$, with  $\iota'=\inf_{x\in I}G_{c+c_\star}(x)>0$, it follows from \eqref{eq:G_var2} that 
\begin{equation*}
\begin{split}
\big\|\Psi_{G_{c+c_{t-p}}}(\mu_{p})-\Psi_{G_{c+c_{t-p}}}\big(\hat{\mu}_{p-1}M\big)\Big\|\leq  \frac{5 \Delta_{T}}{\iota'\iota_{T-1}}
\end{split}
\end{equation*}
which, together with \eqref{eq:p2}, shows that 
\begin{equation}\label{eq:p3}
\begin{split}
\sup_{a\in\R}\Big|\Psi_{G_{c+c_{t-p}}} (\mu_{p} )(h_{a,t-p})-\Psi_{G_{c+c_{t-p}}}\big(  \hat{\mu}_{p-1}M \big)&(h_{a,t-p})\Big|\leq  \frac{5\Delta_{T}}{\iota' \iota_{T-1}}.
\end{split}
\end{equation}

On the other hand, by   Lemma \ref{lemma:h_function} we have
\begin{align*}
|h_{a,k}(x)|\leq \frac{|\rho_{k}x|}{\sigma_{k}\sqrt{2\pi}},\quad\forall (a,x)\in\R^2,\quad\forall k\geq 1 
\end{align*}
and therefore, noting that $|x| G_a(x)\leq a^{-1/2}e^{-1/2}$ for all $x\in\R$ and all $a\in (0,\infty)$,  
\begin{align*}
 G_{c+c_{t-p}}(x)|h_{a,t-p}(x)|\leq   \frac{ |\rho_{t-p}| e^{-1/2}}{\sigma_{t-p}\sqrt{(c+c_{t-p})2\pi} } \leq  \frac{|\rho_{t-p}|}{\sigma_1\,\sqrt{c}  },\quad\forall a\in\R
\end{align*}
where the second inequality uses the fact that $\inf_{k\geq 1}\sigma^2_{k}=\sigma^2_1=\sigma^2/(1+c\sigma^2)>0$ while $\inf_{k\geq 1} c_k\geq 0$ by    Lemma \ref{lemma:Q}.

Using the latter result  and recalling that  $c_\star=\rho^2/\sigma^2\geq \sup_{k\geq 1} c_k$, it follows that we have both
\begin{align*}
\sup_{a\in\R}\Big|\Psi_{G_{c+c_{t-p}}}( \mu_{p})(h_{ a,t-p})\Big|\leq \frac{|\rho_{t-p}|}{\sigma_1\,\sqrt{c}\, \mu_p(G_{c+c_{t-p}}) }\leq \frac{|\rho_{t-p}| }{\sigma_1\,\sqrt{c}\, \mu_p(G_{c+c_{\star}}) }
\end{align*}
and 
\begin{align*}
\sup_{a\in\R}\Big|\Psi_{G_{c+c_{t-p}}} \big(\hat{\mu}_{p-1}M\big)(h_{a,t-p})\Big|\leq \frac{|\rho_{t-p}| }{\sigma_1\,\sqrt{c}\,  \hat{\mu}_{p-1}M(G_{c+c_{t-p}}) }\leq \frac{|\rho_{t-p}| }{\sigma_1\,\sqrt{c}\,  \hat{\mu}_{p-1}M(G_{c+c_{\star}}) } 
\end{align*}
where   
 $\hat{\mu}_{p-1}M(G_{c+c_{\star}})\geq \iota''\iota_{T-1}$ with $\iota''=\eta_1(G_{c+c_{\star}})\inf_{x\in I}G(x)M(x,G_{c+c_\star})$. (This definition of $\iota''$ allows to cover both the case where $p=1$ and the case where $p>1$.) Consequently, if we  define the constant $\bar{C}=\max\{5, 2/( c \sigma_1^2)^{1/2}\}$, then
\begin{equation}\label{eq:p4}
\begin{split}
\sup_{a\in\R} \Big| \Psi_{G_{c+c_{t-p}}}( \mu_{p})(h_{ a,t-p})  -\Psi_{G_{c+c_{t-p}}} \big( \hat{\mu}_{p-1}M\big)(h_{a,t-p})\Big|  \leq  \frac{\bar{C}}{\iota''\iota_{T-1}} |\rho_{t-p}|.
\end{split}
\end{equation}

By combining  \eqref{eq:key_term}, \eqref{eq:p3} and \eqref{eq:p4}, and recalling that $\bar{C}\geq 5$, we obtain that
\begin{equation}\label{eq:p5}
\begin{split}
\sup_{a\in\R}\big|\phi^{t-p}(\hat{\mu}_p)(f_a)- \phi^{t-p} \big( \phi_{p}( \hat{\mu}_{p-1})\big)(f_a)\big|&\leq  \frac{\bar{C}}{\iota''' \iota_{T-1}}\,|\rho_{t-p}|^\gamma  \Delta_{T} ^{1-\gamma},\quad\forall \gamma\in[0,1]
\end{split}
\end{equation}
with $\iota''=\min\{\iota',\iota''\}>0$.

The result in \eqref{eq:p5}   holds for all  $t\in\{2,\dots,T\}$ and all $p\in\{1,\dots,t-1\}$ (assuming that   $T\geq 3$), and we now show that it also holds when $p=t$ for some $t\in\{1,\dots,T\}$, and with $T\geq 1$ arbitrary.  To do so we  let $\gamma\in[0,1]$, $t\in\{1,\dots,T\}$ and $p=t$,   and we note that
\begin{align*}
\sup_{a\in\R}\big|\phi^{t-p}(\hat{\mu}_p)(f_a)- \phi^{t-p} \big( \phi(\hat{\mu}_{p-1})\big)(f_a)\big|&=\sup_{a\in\R}\big| \hat{\mu}_t(f_a)-   \phi( \hat{\mu}_{t-1})(f_a)\big|\\
&= \sup_{a\in\R}\big| \Psi_G(\mu_t)(f_a)-  \Psi_G( \hat{\mu}_{t-1}M ) (f_a)\big|\\
&\leq \frac{5\|\mu_t-\hat{\mu}_{t-1}M\|}{\hat{\mu}_{t-1}M(G_c)}\\
&\leq \frac{5 \Delta_{T}}{\hat{\mu}_{t-1}M(G_{c+c_\star})}
\end{align*}
where  the first inequality holds by Corollary \ref{cor:IS} and uses \eqref{eq:G_var}, and where the last inequality uses the fact that $c_\star=\rho^2/\sigma^2\geq 0$. 

Recalling that $\hat{\mu}_{k}M(G_{c+c_\star})\geq \iota''\iota_{T-1}\geq \iota'''\iota_{T-1}$ for all $k\in\{0,\dots,T-1\}$,  and   noting that  since we are assuming that $p=t$ we have $\rho_{t-p}=\rho_0=1$, it follows that
\begin{align*}
\sup_{a\in\R}\big|\phi^{p-t}(\hat{\mu}_p)(f_a)- \phi^{p-t} \big( \phi (\hat{\mu}_{p-1})\big)(f_a)\big|&\leq \bar{C} (\iota'''\iota_{T-1})^{-1}  \min\big\{1, \Delta_{ T} \big\}\\
&\leq \bar{C}(\iota'''\iota_{T-1})^{-1}  \min\big\{1,\Delta_{ T}^{1-\gamma}\big\}\\
&\leq \bar{C}(\iota'''\iota_{T-1})^{-1}  |\rho_{t-p}|^\gamma \Delta_{ T}^{1-\gamma} 
\end{align*}
showing that \eqref{eq:p5} holds for any integers $1\leq p\leq t\leq T$.

We now  let $C_\star\in(0,\infty)$ and $\epsilon_\star\in(0,1)$ be as in Lemma \ref{lemma:Q}, so that   $|\rho_k|\leq C_\star \epsilon_\star^k$ for all $k\geq 0$. In addition, we let $t\in\{1,\dots,T\}$ and $k\in\mathbb{N}$.  Then,  by using \eqref{eq:decomp1} and  \eqref{eq:p5}, and using the convention that empty sums equal zero, we have
 \begin{equation}\label{eq:final}
 \begin{split}
\|\hat{\mu}_t-\hat{\eta}_t\|&  = \sup_{a\in\R} \big|\hat{\mu}_t(f_a)-\hat{\eta}_t(f_a)\big|\\
&\leq   \sum_{p=1}^t\sup_{a\in\R}\Big|\phi^{t-p}(\hat{\mu}_p)(f_a)-\phi^{t-p}\big( \phi(\hat{\mu}_{p-1})\big)(f_a)\Big|\\
 &=  \sum_{p=1}^{t-k}\sup_{a\in\R}\Big|\phi^{t-p}(\hat{\mu}_p)(f_a)-\phi^{t-p}\big( \phi(\hat{\mu}_{p-1})\big)(f_a)\Big| \\
 &+  \sum_{p=t-k+1}^t\sup_{a\in\R}\Big|\phi^{t-p}(\hat{\mu}_p)(f_a)-\phi^{t-p}\big( \phi(\hat{\mu}_{p-1})\big)(f_a)\Big|\\
 &\leq \frac{\bar{C}}{\iota''' \iota_{T-1}}\bigg(   C_\star \sum_{p=1}^{t-k}\epsilon_\star^{t-p}+ k\Delta_{ T}\bigg)\\
&\leq \frac{\bar{C}}{\iota''' \iota_{T-1}}\bigg(\epsilon_\star^k   C_\star (1-\epsilon_\star)^{-1}+ k\Delta_{ T} \bigg).
 \end{split}
 \end{equation}
By applying  \eqref{eq:final} with
 \begin{align*}
 k=\bigg\lceil \frac{1}{\log(1/\epsilon_\star)}\log\bigg(\frac{C''\log(1/\epsilon_\star)}{\Delta_{T}}\bigg)\bigg\rceil ,\quad C'':=   C_\star (1-\epsilon_\star)^{-1},
 \end{align*}
  we obtain that \begin{align*}
\|\hat{\mu}_t  -\hat{\eta}_t\|&\leq \Delta_{ T}\,\bigg(\frac{\bar{C}}{\iota'''\iota_{T-1}  \log(1/\epsilon_\star)}\bigg)\bigg(2 +\log(C'')+\log\big(\log(1/\epsilon_\star)\big) - \log(\epsilon_\star\Delta_{ T}) \bigg)\\
&\leq C''' \Delta_{ T}\log(1+\Delta_{ T}^{-1})
\end{align*}
where the second inequality holds for some constant $C'''\in(0,\infty)$ depending only on   $C''$ and $\epsilon_\star$, and uses the fact that $\Delta_{ T}\in [0,1]$. The first part of the lemma follows.

To show the second part of the lemma assume first that $T\geq 2$, let $t\in\{ 2,\dots,T\}$, and note that  
\begin{align}\label{eq:p2p2}
\|\mu_t-\eta_t\|&\leq \|\mu_t -\hat{\mu}_{t-1}M\|+\|\hat{\mu}_{t-1} M-\hat{\eta}_{t-1}M\|\leq \Delta_{T}+\|\hat{\mu}_{t-1} M-\hat{\eta}_{t-1}M\|. 
\end{align}
Noting that for all $a\in\R$ we have
\begin{align*}
M(f_a)(x)=\Phi\bigg(\frac{a- \rho  x}{\sigma }\bigg)=:\tilde{g}_{a}(x),\quad\forall x\in\R
\end{align*}
where, using Lemma \ref{lemma:h_function}, $V(\tilde{g}_{a})=V(h_{a,\rho,\sigma})\leq 1$, it follows from Lemma \ref{lemma:KH_in} that 
\begin{align*}
\|\hat{\mu}_{t-1} M-\hat{\eta}_{t-1}M\|=\sup_{a\in\R}|\hat{\mu}_{t-1}(\tilde{g}_a)-\hat{\eta}_{t-1}(\tilde{g}_a)|\leq \|\hat{\mu}_{t-1}-\hat{\eta}_{t-1}\|.
\end{align*}
By combining this latter result with \eqref{eq:p2p2} we obtain that
\begin{align*}
\sup_{t\in\{2,\dots,T\}}\|\mu_t -\eta_t\|\leq \Delta_{ T}+ \sup_{t\in\{1,\dots,T\}} \|\hat{\mu}_{t}-\hat{\eta}_{t}\|.
\end{align*}
Then, the second part of  the lemma follows from the first part of the lemma, upon noting that $\|\hat{\mu}_1-\eta_1\|=\|\hat{\mu}_1- \hat{\mu}_{0}^N M\|\leq   \Delta_{T}$. This ends the proof of the lemma.
\end{proof}

\subsubsection{Proof of the lemma}

\begin{proof}[Proof of Lemma \ref{lemma:key1}]

Let $(\mu_t)_{t\geq 1}$ and $(\hat{\mu}_t)_{t\geq 1}$ be as in the statement of the lemma, and let $I=[-\xi,\xi]$ for some $\xi\in(0,\infty)$  such that $\eta_1(I)>0$. Then, by using Lemma \ref{lemma:Q}, we  can easily verify verify that $\iota_1:=\inf_{t\geq 1}\hat{\eta}_t(I)>0$, which implies that  $\iota_2:=\inf_{t\geq 2}\eta_t(I)\geq \iota_1\inf_{x\in I}M(x,I)>0$. Consequently, there exists a constant $\iota\in(0,1)$ such that
$\inf_{t\geq 1}\eta_t(I)\geq \iota$.

To prove the lemma   we let $\tilde{\Delta}_{T}=\sup_{t\in\{1,\dots,T\}}\|\mu_t-\hat{\mu}_{t-1}M\|_{\mathrm{D}}$ for all $T\geq 1$,  with the convention    that $\hat{\mu}_{0}M=\eta_1$, and we remark first that, for all $t\geq 1$,
\begin{align}\label{eq:useDelta}
\|\mu_t-\eta_t\|_{\mathrm{D}}\leq \iota/2\implies  \mu_t(I)=\eta_t(I)+\mu_t(I)-\eta_t(I)\geq \eta_t(I)-\|\mu_t-\eta_t\|_{\mathrm{D}}\geq  \iota/2.
\end{align}
Then, by using \eqref{eq:KD},  \eqref{eq:useDelta} and Lemma \ref{lemma:key0},   to prove the lemma it suffices to show that   there exists an $\epsilon\in (0,1)$ such that
\begin{align}\label{eq:show_imp}
\tilde{\Delta}_{T}\leq\epsilon\implies \sup_{t\in\{1,\dots,T\}}\|\mu_t-\eta_t\|_{\mathrm{D}}\leq \iota/2,\quad\forall T\geq 1.
\end{align}
To do so   we let $C_I\in [1,\infty)$ be as in the statement of Lemma \ref{lemma:key0} (when $I=[-\xi,\xi]$), $\delta\in (0,1)$ be arbitrary, and   $C'=\max\big\{1,\sup_{x\in (0,1]} x^{1-\delta}\log(1+1/x)\big\}<\infty$. With this notation in place, we show by induction on $T\geq 1$ that  \eqref{eq:show_imp} holds for $\epsilon=(\iota^2/(8 C_I\, C'))^{1/\delta}$.

Since $\epsilon\leq \iota/2$ and $\tilde{\Delta}_1=\|\mu_1-\eta_1\|_{\mathrm{D}}$, the implication in \eqref{eq:show_imp}   holds when $T=1$, and we now assume that it holds when $T=s$ for some $s\geq 1$. Then, by using   \eqref{eq:useDelta}, it follows that $\mu_t(I)\geq \iota/2$ for all $t\in\{1,\dots,s\}$ and thus, by   Lemma \ref{lemma:key0} and using \eqref{eq:KD}, we have
\begin{align}\label{eq:dis1}
\sup_{t\in\{1,\dots,s+1\}} \|\mu_t- \eta_t\|_{\mathrm{D}}\leq 2 \sup_{t\in\{1,\dots,s+1\}} \|\mu_t-\eta_t\|_{\mathrm{D}} \leq \frac{ 4C_I}{\iota} \Delta_{s+1} \log(1+\Delta_{s+1}^{-1}) 
\end{align}
where
\begin{align}\label{eq:dis2}
\frac{ 4C_I}{\iota} \Delta_{s+1} \log(1+\Delta_{s+1}^{-1})\leq  \frac{\iota}{2} \Leftrightarrow  \Delta_{s+1} \log(1+\Delta_{s+1}^{-1})\leq \frac{\iota^2}{8 C_I}.
\end{align}
By the definition of the constant $C'$, we have $\Delta_{s+1} \log(1+\Delta_{s+1}^{-1})\leq C'\,  \Delta_{s+1}^{\delta}\leq  C'\,  \tilde{\Delta}_{s+1}^{\delta}$ which, together with \eqref{eq:dis1} and \eqref{eq:dis2}, shows that $\tilde{\Delta}_{s+1}\leq\epsilon\implies \sup_{t\in\{1,\dots,s+1\}}\|\mu_t-\eta_t\|_{\mathrm{D}}\leq \iota/2$, showing that \eqref{eq:show_imp} holds for $T=s+1$. This concludes the proof of \eqref{eq:show_imp} and thus of the lemma.

\end{proof}

\bibliographystyle{apalike}
\bibliography{complete}

\end{document}